\documentclass[twoside]{aiml}

\usepackage{aimlmacro}
\usepackage{graphicx}
\usepackage{amsmath}
\usepackage{amssymb}
\usepackage{graphicx}

\makeatletter
\providecommand{\leftsquigarrow}{%
  \mathrel{\mathpalette\reflect@squig\relax}%
}
\newcommand{\reflect@squig}[2]{%
  \reflectbox{$\m@th#1\rightsquigarrow$}%
}
\makeatother

\newcommand{\ydcolor}{\color{blue}}
\newcommand{\ydprefix}{YD: }
\newcommand{\ydnote}[1]{{\ydcolor \ydprefix #1 }}

\newcommand{\val}[1]{[\![{#1}]\!]}
\newcommand{\descr}[1]{(\![{#1}]\!)}
\newcommand{\atprop}{\mathcal{V}}

\def\lastname{Ding, Manoorkar, Panettiere, and Wang}

\begin{document}

\begin{frontmatter}
  \title{Toward the van Benthem Characterization Theorem for Non-Distributive Modal Logic}\thanks{This project has received funding from the European Union’s Horizon 2020 research and innovation programme under the Marie Skłodowska-Curie grant agreement No 101007627. The research of Krishna Manoorkar is supported by the NWO grant KIVI.2019.001. Yiwen Ding and Ruoding Wang are supported by the China Scholarship Council.}
  \author{Yiwen Ding$^{\,a}$}
  \author{Krishna Manoorkar$^{\,a}$}
  \author{Mattia Panettiere$^{\,a}$}
  \author{Ruoding Wang$^{\,a, \,b}$}
  \address[1]{Vrije Universiteit Amsterdam, The Netherlands}
  \address[3]{Xiamen University, China}

  \begin{abstract}
 In this paper, we introduce the simulations and bisimulations on polarity-based semantics for non-distributive modal logic, which are natural generalizations of those notions on Kripke semantics for modal logic. We also generalize other important model-theoretic notions about Kripke semantics such as image-finite models, modally-saturated models, ultrafilter extension and ultrapower extension to the non-distributive setting. By using these generalizations, we prove the Hennessy-Milner theorem and the van Benthem characterization theorem for non-distributive modal logic based on polarity-based semantics.
  \end{abstract}

 \end{frontmatter}

\section{Introduction}\label{sec:introduction}

This paper focuses on advancing the model theory of polarity-based semantics for {\em normal non-distributive modal logics} \cite{frittella2020toward,conradie2021rough,conradie2020non}. These logics, also known as {\em LE-logics} (lattice expansion logics), are characterized by algebraic semantics provided by arbitrary lattices expanded with normal operators. Leveraging the interpretation of polarities as formal contexts in Formal Concept Analysis \cite{ganter1997applied}, non-distributive modal logic is conceptualized as a logic of formal concepts or categories enriched with modal operators, such as approximation or knowledge operators \cite{conradie2021rough,conradie2017toward}. Several model theoretic notions for polarity-based semantics like  {\em disjoint unions}, {\em p-morphisms}, and {\em ultrafilter extensions} were introduced in prior work \cite{conradie2018goldblatt}. These concepts were  used in \cite{conradie2018goldblatt} to prove  {\em Goldblatt-Thomason theorem} for LE-logics. Simulations on polarity-based models were defined  and  employed to demonstrate the preservation of lattice-based modal $\mu$-calculus (an extension of basic LE-logic with the least and the greatest fixed point operators) under simulations \cite[Section 6]{ding2023game}. In this  paper,  we delve more  into    the problem of generalizing simulations and  bisimulations from Kripke semantics to polarity-based semantics, and also prove fundamental results like Hennessy-Milner theorem and van Benthem characterization theorem for polarity-based semantics.

Bisimulations were initially defined by van Benthem \cite{van2014modal} as p-relations on Kripke frames for classical modal logic. They aim to characterize points in Kripke models with equivalent behavior concerning modal logic formulas. Since their inception, bisimulations have played an important role in modal logic theory and its applications in computer science \cite{van2014modal,ponse1995modal,hennessy1995symbolic,goranko20075}. Hennessy and Milner \cite{hennessy1985algebraic} showed that on image-finite Kripke models, the semantic equivalence relation for modal logic forms a bisimulation relation, a result later extended to modally-saturated Kripke models \cite{hollenberg1994hennessy}. Van Benthem \cite{van2014modal} established the characterization theorem for modal logic, showing that classical modal logic is the bisimulation-invariant fragment of first-order logic.

Bisimulations have been extensively defined and studied across various semantic models for different logics. For instance, in topological semantics for (S4) modal logic, they are referred to as topo-bisimulations \cite{aiello2003reasoning}, where the Hennessy-Milner property for a class of topological models was established \cite{DAVOREN2009349}. Similarly, in Kripke semantics for intuitionistic modal logic, they are known as asimulations, and results like the Hennessy-Milner theorem and the van Benthem theorem have been demonstrated \cite{olkhovikov2013model}. Bezhanishvilli and Henke \cite{bezhanishvili2020model} proved the van Benthem theorem for models based on general descriptive frames using {\em Vietoris bisimulations} \cite{bezhanishvili2010vietoris}. Similar studies have been carried out for  different semantics of graded modal logics, dynamic modal logics, fuzzy modal logics, monotone modal logics, etc. \cite{de2000note,stankovic2023simulations,fuzzy_bisimul_character,dynamic_bisimulation,hansen2009neighbourhood}. De Groot \cite{de2024non} established a van Benthem theorem for the meet-semilattice semantics of non-distributive (positive) logic using the notion of meet simulations. In this work, we adopt a similar approach to prove the van Benthem theorem for the polarity-based semantics of normal non-distributive modal logic by defining appropriate notion of simulations. Along the way, we introduce essential model-theoretic notions like image-finite models, modally-saturated models, and $\omega$-saturated models for this semantics, and prove significant results like the Hennessy-Milner theorem and the Hennessy-Milner property for the class of modally-saturated models.

{\em Structure of the paper.} In Section \ref{sec:prelim}, we collect the required  preliminaries. In Section \ref{sec:bisimulations}, we define the simulations and bisimulations on the polarity-based semantics and prove the Hennessy-Milner theorem. In Section \ref{sec:Modal Saturation via Filter-Ideal Extensions}, we define the modally-saturated polarity-based models, filter-ideal extensions of polarity-based models and show that two pointed polarity-based models are modally equivalent if and only if their filter-ideal extensions are bisimilar. In Section \ref{sec:Toward van Benthem Characterization Theorem}, we define the standard translation and prove the van Benthem characterization theorem for polarity-based semantics of non-distributive modal logic. In Section \ref{sec:Conclusion}, we summarize the main findings of the paper and suggest some possible future works.

\section{Preliminary}\label{sec:prelim} 
In this section, we gather  useful preliminaries about normal non-distributive modal logic and its polarity-based semantics based on \cite{conradie2020non} and \cite{conradie2017toward}. We assume readers are familiar with the basic concepts from the model theory of classical  modal logic such as simulations, bisimulations, ultrafilter extensions, standard translations, Hennessy-Milner theorem, and van Benthem characterization theorem. For a detailed discussion on the model theory of modal logic, we refer to \cite[Section 2]{Blackburn_Rijke_Venema_2001}. 

\subsection{Non-Distributive Modal Logic}\label{ssec:Non-distributive modal logic}
Let $\atprop$ be a countable set of propositional variables. The language $\mathcal{L}$ (i.e. set of formulas) is defined as follows:

{{
\centering
$\varphi ::= p \mid  \bot \mid \top \mid \varphi \wedge \varphi \mid \varphi \vee \varphi \mid \Box \varphi \mid \Diamond \varphi$,
\par
}}
   
\noindent
where $p\in \atprop$. Given any $\varphi,\psi\in\mathcal{L}$, $\varphi\vdash\psi$ is an {\em $\mathcal{L}$-sequent}. A {\em normal non-distributive modal logic} is a set of $\mathcal{L}$-sequents containing the following axioms:

{{\centering 
$p \vdash p,
\ \
p \vdash \top
\ \
\bot \vdash p,
\ \
p \vdash p \vee q,
\ \
q \vdash p \vee q,
\ \
p \wedge q \vdash p,
\ \
p \wedge q \vdash q,
\ $
\\
$\top \vdash \Box\top,
\ \
\Diamond\bot \vdash \bot,
\ \
\Box p \wedge \Box q \vdash \Box(p \wedge q),
\ \
\Diamond(p \vee q) \vdash \Diamond p \vee \Diamond q
\ $
\par}}
\smallskip

\noindent and closed under the following inference rules:

\smallskip
{{
\centering
$\frac{\varphi\vdash \chi\quad \chi\vdash \psi}{\varphi\vdash \psi}
\ \ 
\frac{\varphi\vdash \psi}{\varphi\left(\chi/p\right)\vdash\psi\left(\chi/p\right)}
\ \ 
\frac{\chi\vdash\varphi\quad \chi\vdash\psi}{\chi\vdash \varphi\wedge\psi}
\ \ 
\frac{\varphi\vdash\chi\quad \psi\vdash\chi}{\varphi\vee\psi\vdash\chi}
\ \ 
\frac{\varphi\vdash\psi}{\Box \varphi\vdash \Box \psi}
\ \ 
\frac{\varphi\vdash\psi}{\Diamond \varphi\vdash \Diamond \psi}$
\par
}}
\smallskip 

The smallest such set $\mathbb{L}$ is called the {\em basic} normal non-distributive modal logic. For any sequent $\varphi\vdash\psi$ in $\mathbb{L}$ we denote it as $\varphi\vdash_{\mathbb{L}}\psi$.

\subsection{Polarity-Based Semantics}

Given any relation $R\subseteq A \times X$ and any $B \subseteq A$, $Y \subseteq X$, the maps $R^{(1)}: \mathcal{P}(A)\rightarrow\mathcal{P}(X)$ and $R^{(0)}:\mathcal{P}(X)\rightarrow\mathcal{P}(A)$ are defined as $R^{(1)}[B] := \{x \in X \mid \forall b (b \in B \Rightarrow b R x)\}$ and $R^{(0)}[Y] := \{a \in A \mid \forall y (y \in Y \Rightarrow a R y)\}$.

\begin{definition}
A {\em polarity} is a tuple $\mathfrak{P}=(A,X,I)$ such that $A$ and $X$ are sets and $I\subseteq A\times X$ is a binary relation. Given a polarity $\mathfrak{P}=(A,X,I)$, the maps $(\cdot)^\uparrow :\mathcal{P}(A)\rightarrow\mathcal{P}(X)$ and $(\cdot)^\downarrow :\mathcal{P}(X)\rightarrow\mathcal{P}(A)$ are defined as follows:

{{
\centering
 $B^\uparrow := I^{(1)}[B]$ and $Y^\downarrow := I^{(0)}[Y] $, where $B\subseteq A$ and $Y\subseteq X$.
 \par
}}
\noindent
 The maps $(\cdot)^\uparrow$ and $(\cdot)^\downarrow$ form a {\em Galois connection} between posets $(\mathcal{P}(A),\subseteq)$ and $(\mathcal{P}(X),\subseteq)$, that is $Y\subseteq B^\uparrow$ iff $B\subseteq Y^\downarrow$ for all $B\in \mathcal{P}(A)$ and $Y\in \mathcal{P}(X)$.
\end{definition}

\begin{definition}
    A {\em formal concept} of the polarity $\mathfrak{P}=(A,X,I)$ is a pair $(\val{c},\descr{c})$ such that $\val{c}\subseteq A$, $\descr{c}\subseteq X$, and $\val{c}^\uparrow = \descr{c}$, $\descr{c}^\downarrow = \val{c}$. It follows that $\val{c}$ and $\descr{c}$ are {\em Galois-stable}, i.e.\ $\val{c}^{\uparrow\downarrow} = \val{c}$ and $\descr{c}^{\downarrow\uparrow} = \descr{c}$. The set $\mathcal{C}(\mathfrak{P})$ of all the formal concepts of $\mathfrak{P}$ can be partially ordered as follows: for any $c,d\in \mathcal{C}(\mathfrak{P})$, $c\leq d$ iff $\val{c}\subseteq \val{d}$ iff $\descr{d}\subseteq \descr{c}$.
    The poset $\mathfrak{P}^+=(\mathcal{C}(\mathfrak{P}), \leq)$ is a complete lattice -- called the {\em concept lattice} of $\mathfrak{P}$ -- such that for any $\mathrm{K}\subseteq \mathcal{C}(\mathfrak{P})$,
    
    {{
    \centering
    $\bigwedge\mathrm{K} = (\bigcap\{\val{c}\mid c\in\mathrm{K}\},(\bigcap\{\val{c}\mid c\in\mathrm{K}\})^\uparrow)$,
    
    $\bigvee\mathrm{K} = ((\bigcap\{\descr{c}\mid c\in\mathrm{K}\})^\downarrow,\bigcap\{\descr{c}\mid c\in\mathrm{K}\})$.
    \par 
    }}
\end{definition}

\begin{proposition}\label{prop:concept lattice}
For any polarity $\mathfrak{P} = (A,X,I)$, the concept lattice $\mathfrak{P}^+$ is completely join-generated by the set $\{\textbf{a} := (a^{\uparrow\downarrow},a^\uparrow)\mid a\in A\}$, and is completely meet-generated by the set $\{\textbf{y} := (y^\downarrow,y^{\downarrow\uparrow})\mid y\in Y\}$.
\end{proposition}

\begin{theorem}[Birkhoff's representation theorem]
Any complete lattice $\mathbf{L}$ is isomorphic to the concept lattice $\mathfrak{P}^+$ of some polarity $\mathfrak{P}$.
\end{theorem}

\begin{definition}
    An {\em $\mathrm{LE}$-frame} is a tuple $\mathfrak{F} = (\mathfrak{P}, R_\Diamond, R_\Box)$, where $\mathfrak{P}=(A,X,I)$ is a polarity, and $R_\Diamond\subseteq X\times A$ and  $R_\Box\subseteq A\times X$ are {\em I-compatible} relations, i.e., for all $a\in A$ and $x\in X$, $R^{(0)}_\Diamond [a]$ (resp. $R^{(0)}_\Box [x]$) and $R^{(1)}_\Box [a]$ (resp. $R^{(1)}_\Diamond [x]$) are Galois-stable.
\end{definition}

\begin{definition}
    For any $\mathrm{LE}$-frame $\mathfrak{F}=(\mathfrak{P}, R_\Diamond, R_\Box)$, the {\em complex algebra} of $\mathfrak{F}$ is $\mathfrak{F}^+:=(\mathfrak{P}^+, \Diamond^{\mathfrak{P}^+}, \Box^{\mathfrak{P}^+})$ where $\mathfrak{P}^+$ is the concept lattice of $\mathfrak{P}$, $\Diamond^{\mathfrak{P}^+}$ and $\Box^{\mathfrak{P}^+}$ are unary operators on $\mathfrak{P}^+$ defined as follows: for every $c\in \mathfrak{P}^+$,
    \smallskip
    
    {{\centering
    $\Diamond^{\mathfrak{P}^+}(c) := (R^{(0)}_\Diamond [\val{c}]^\downarrow, R^{(0)}_\Diamond [\val{c}])$\quad and \quad
    $\Box^{\mathfrak{P}^+}(c) := (R^{(0)}_\Box [\descr{c}], R^{(0)}_\Box [\descr{c}]^\uparrow)$.
    \par}}
    \smallskip
    
\noindent Note that $\Diamond^{\mathfrak{P}^+}$ (resp.\ $\Box^{\mathfrak{P}^+}$) is completely join (resp.\ meet) preserving.
\end{definition}
\begin{definition}
A {\em valuation} on an $\mathrm{LE}$-frame $\mathfrak{F} = (\mathfrak{P}, R_{\Diamond}, R_{\Box})$ is a map $V: \atprop \rightarrow\mathfrak{P}^+$. For each $p\in\atprop$, we let $\val{p}: = \val{V(p)}$ (resp.~$\descr{p}: = \descr{V(p)}$) denote the extension (resp.~intension) of the interpretation of $p$ under $V$. A valuation can be homomorphically extended to an unique valuation $\overline V: \mathcal{L} \rightarrow\mathfrak{P}^+$ on all the $\mathcal{L}$-formulas. An {\em $\mathrm{LE}$-model} is a pair $(\mathfrak{F}, V)$, where $\mathfrak{F}$ is an $\mathrm{LE}$-frame, and $V$ is a valuation on it.

    %Let $L$ be a set of all formulas in $\mathbf{L}$. For an arbitrary polarity $\mathfrak{P}$, any homomorphic assignment $\overline{v}:L\rightarrow\mathfrak{P}^+$ will give rise to pairs of relations $(\Vdash,\succ)$ such that $\Vdash\,\subseteq G\times L$, $\succ\,\subseteq M\times L$, then for every $g\in G$, $m\in M$ and every formula $\varphi\in L$:
\end{definition}

\begin{definition}
For any $\mathrm{LE}$-model $\mathfrak{M} = (\mathfrak{F}, V)$, the {\em modal satisfaction} relations $\Vdash$ and $\succ$ are defined inductively as follows:

\smallskip
{{\centering 
\footnotesize{
\begin{tabular}{l@{\hspace{1em}}l@{\hspace{1.2em}}l@{\hspace{1em}}l}
$\mathfrak{M}, a \Vdash p$ & iff $a\in \val{p}_{\mathfrak{M}}$ &
$\mathfrak{M}, x \succ p$ & iff $x\in \descr{p}_{\mathfrak{M}}$ \\
$\mathfrak{M}, a \Vdash\top$ & always &
$\mathfrak{M}, x \succ \top$ & iff   $(\forall a\in A) aIx$\\
$\mathfrak{M}, x \succ  \bot$ & always &
$\mathfrak{M}, a \Vdash \bot $ & iff $(\forall x\in X) aIx$ \\
$\mathfrak{M}, a \Vdash \varphi\wedge \psi$ & iff $\mathfrak{M}, a \Vdash \varphi$ and $\mathfrak{M}, a \Vdash  \psi$ & 
$\mathfrak{M}, x \succ \varphi\wedge \psi$ & iff $(\forall a\in A)$ $(\mathfrak{M}, a \Vdash \varphi\wedge \psi \Rightarrow a I x)$
\\
$\mathfrak{M}, x \succ \varphi\vee \psi$ & iff  $\mathfrak{M}, x \succ \varphi$ and $\mathfrak{M}, x \succ  \psi$ & 
$\mathfrak{M}, a \Vdash \varphi\vee \psi$ & iff $(\forall x\in X)$ $(\mathfrak{M}, x \succ \varphi\vee \psi \Rightarrow a I x)$\\
$\mathfrak{M}, a \Vdash \Box\varphi$ &  iff $(\forall x\in X)(\mathfrak{M}, x \succ \varphi \Rightarrow a R_\Box x)$ &
$\mathfrak{M}, x \succ \Box\varphi$ &  iff $(\forall a\in A)(\mathfrak{M}, a \Vdash \Box\varphi \Rightarrow a I x)$\\
$\mathfrak{M}, x \succ \Diamond\varphi$ &  iff $(\forall a\in A)(\mathfrak{M}, a \Vdash \varphi\Rightarrow xR_\Diamond a)$ &
$\mathfrak{M}, a \Vdash \Diamond\varphi$ & iff $(\forall x\in X)(\mathfrak{M}, x \succ \Diamond\varphi \Rightarrow a I x)$ \\
\end{tabular}
}
\par}}
\smallskip
Note that unlike the classical modal logic, the modal operators $\Box$ and $\Diamond$ are not inter-definable in LE-logic. 
%For any $\varphi\in\mathcal{L}$, $\mathfrak{M}\Vdash\varphi$ if for any $a\in A$, $\mathfrak{M}, a\Vdash\varphi$, and $\mathfrak{M}\succ\varphi$ if for any $x\in X$, $\mathfrak{M}, x\succ\varphi$.
For any $\mathcal{L}$-sequent $\varphi \vdash \psi$, and an $\mathrm{LE}$-model $\mathfrak{M} = (\mathfrak{F},V)$, $\mathfrak{M} \vDash \varphi \vdash \psi$ iff $\val{V(\varphi)} \subseteq \val{V(\psi)} $ iff $\descr{V(\psi)} \subseteq \descr{V(\varphi)}$. 

\end{definition}
\noindent The following theorem states that the basic normal non-distributive modal logic defined in Section \ref{ssec:Non-distributive modal logic} is sound and complete w.r.t. the class of $\mathrm{LE}$-models.

\begin{theorem}[Proposition 3, \cite{conradie2017toward}]
For any unprovable $\mathcal{L}$-sequent $\varphi\vdash\psi$ in the non-distributive modal logic, there is an $\mathrm{LE}$-model $\mathfrak{M}$ such that $\mathfrak{M}\nvDash\varphi\vdash\psi$.
\end{theorem}

\subsection{Two Sorted First-Order Logic}\label{Two Sorted First-Order Logic}
  Given a countable set of propositional variables $\atprop$, we let $\mathcal{L}^1$ be the two sorted first-order language with equality built over two disjoint countable sets of variables $G$ and $M$, three binary predicates $I$, $R_\Box$, $R_\Diamond$, and two unary predicates $P_A$ and $P_X$ for every $p\in\atprop$. A {\em domain} for $\mathcal{L}^1$ is a pair of disjoint sets $A$ and $X$. An {\em interpretation} $\mathrm{I}$ of  $\mathcal{L}^1$ over domain $(A,X)$ assigns binary predicates $I$, $R_\Box$, and $R_\Diamond$ to  relations  $I^\mathrm{I} \subseteq A \times X$, ${R^\mathrm{I}}_\Box \subseteq A \times X$,  and ${R^\mathrm{I}}_\Diamond \subseteq X \times A$, such that ${R^\mathrm{I}}_\Box $ and ${R^\mathrm{I}}_\Diamond $ are $I^\mathrm{I}$-compatible. For every propositional variable $p\in\atprop$, the unary  predicates 
$P_A$ and $P_X$ are assigned to sets $P^{\mathrm{I}}_A \subseteq A$, $P^{\mathrm{I}}_X \subseteq X$, such that, $(I^\mathrm{I})^{(1)}[P^{\mathrm{I}}_A]=P^{\mathrm{I}}_X$, and $(I^\mathrm{I})^{(0)}[P^{\mathrm{I}}_X]=P^{\mathrm{I}}_A$, respectively. An $\mathcal{L}^1$-{\em structure} is a tuple $\mathfrak{M}=(A, X, \mathrm{I})$, where $\mathrm{I}$ is an interpretation of $\mathcal{L}^1$ over domain $A$ and $X$. A {\em valuation} on an $\mathcal{L}^1$- structure $\mathfrak{M}$ is a map $v$ which assigns every variable $g\in G$ (resp.~$m\in M$) to an element $a\in A$ (resp.~$x\in X$). The {\em satisfaction relation} for any $\mathcal{L}^1$-term w.r.t.~any $\mathcal{L}^1$- structure $\mathfrak{M}$ and valuation $v$ on it is defined as follows: 

\smallskip
{{
\centering
\small{
\begin{tabular}{ll}
     1.~$\mathfrak{M}, v\vDash g_1=g_2$ iff $v(g_1)=v(g_2)$ & 2.~$\mathfrak{M},v\vDash m_1=m_2$ iff $v(m_1)=v(m_2)$\\
      3.~$\mathfrak{M}, v\vDash P_A(g_1)$ iff $v(g_1)\in P_A^{\mathrm{I}} $ & 4.~$\mathfrak{M},v\vDash P_X(m_1)$ iff $v(m_1)\in P_X^{\mathrm{I}}$ \\
      5.~$\mathfrak{M},v\vDash g_1Im_1$ iff $v(g_1)I^{\mathrm{I}} v(m_1)$ & 6.~$\mathfrak{M}, v\vDash g_1 R_\Box m_1$ iff $v(g_1)R_\Box^{\mathrm{I}} v(m_1)$ \\
      7.~$\mathfrak{M},v \vDash m_1 R_\Diamond g_1$ iff $v(m_1)R_\Diamond^{\mathrm{I}} v(g_1)$. 
      \end{tabular}
\par
}
}}
\smallskip

Any $\mathrm{LE}$-model $\mathfrak{M}=(A, X, I^{\mathfrak{M}}, R^{\mathfrak{M}}_\Box, R^{\mathfrak{M}}_\Diamond, V)$ can be seen as an $\mathcal{L}^1$-structure  (and vice versa), where  the domain of interpretation is formed by sets $A$ and $X$ from the polarity. The binary predicate symbols  $I$, $R_\Box$, $R_\Diamond$ are interpreted by the corresponding relations on the polarity, and unary predicate $P_A$ (resp.~$P_X$) is interpreted as the set $\val{V(p)}$ (resp.~$\descr{V(p)}$), for any  $p \in \atprop$.  For any $\mathrm{LE}$-model $\mathfrak{M}$, we use the same name $\mathfrak{M}$ for the two sorted first-order structure when they are clearly distinguishable from the context. 

\iffalse 
In a each $P_G$ and $P_M$ are interpreted as $\val{p}$ and $\descr{p}$ respectively, and $I$, $R_\Box$ and $R_\Diamond$ are interpreted as $I^{\mathfrak{M}}$, $R^{\mathfrak{M}}_\Box$ and $R^{\mathfrak{M}}_\Diamond$ respectively. An {\em assignment} on $\mathfrak{M}$ is a pair of maps $(v_G, v_M)$ such that $v_G: G\rightarrow A$ and $v_M: M\rightarrow X$. Given variables $g_1,g_2\in G$ and $m_1,m_2\in M$, the {\em satisfaction} on $\mathfrak{M}$ under assignment $(v_A, v_X)$ is defined as following:

Let $\varphi_g$ and $\varphi_m$ be $\mathcal{L}_1$-formulas where $g\in G$ and $m\in M$ are their only free variables respectively. For any $a\in A$ and $x\in X$, we use $\mathfrak{M}\vDash\varphi_g [a]$ denotes $\mathfrak{M}, (v_G, v_M)\vDash\varphi_a$ where $v_G(g)=a$, and $\mathfrak{M}\vDash\varphi_m [x]$ denotes $\mathfrak{M}, (v_G, v_M)\vDash\varphi_m$ where $v_M(m)=x$. 
\fi

\section{Simulations and Bisimulations on \texorpdfstring{$\mathrm{LE}$}{LE}-Models}\label{sec:bisimulations}
%In this section, we generalize the bisimulations from Kripke semantics to the polarity-based semantics. The first natural definition, you can come with is as follows: 
In this section, we generalize simulations and bisimulations from  Kripke models to $\mathrm{LE}$-models and show invariance of $\mathcal{L}$-formulas under bisimulations. We give an example which shows that the  Hennesy-Milner theorem does not hold with respect to this definition of bisimulation. Hence, we give a more general definition of bisimilarity which is based on simulations, and  prove Hennesy-Milner theorem for polarity-based semantics of non-distributive modal logic with respect to this generalized notion of bisimilarity.

%We define a natural generalization of bis, 
%ive two definitions of bisimulation on $\mathrm{LE}$-models, and then show that basic non-distributive modal logic is invariant the bisimulation. Next, we investigate the relationship between modal satisfaction and bisimulations. We generalize the definition of image-finite from Kripke model to $\mathrm{LE}$-model, and show the Hannesy-Milner Theorem for it. 

%we define and show basic properties of bisimulations on polarity-based models for polarities.

\begin{definition}\label{def: modally equivalent}
Let $\mathfrak{M}_1=(\mathfrak{F}_1, V_1)$ and $\mathfrak{M}_2=(\mathfrak{F}_2, V_2)$ be any $\mathrm{LE}$-models, $a_1, x_1$ in $\mathfrak{M}_1$ and $a_2, x_2$ in $\mathfrak{M}_1$, then

%where $\mathfrak{F}_1=(A_1,X_1,I_1,{R_\Box}_1, {R_{\Diamond}}_1)$ and $\mathfrak{F}_2=(A_2,X_2,I_2,{R_\Box}_2, {R_{\Diamond}}_2)$. Let $a_1\in A_1$, $x_1\in X_1$, $a_2\in A_2$ and $x_2\in X_2$.
1.~$\mathfrak{M}_1, a_1\rightsquigarrow_A \mathfrak{M}_2, a_2$ if for any $\varphi\in\mathcal{L}$, $\mathfrak{M}_1, a_1\Vdash \varphi$ implies $\mathfrak{M}_2, a_2\Vdash \varphi$.

 2.~$\mathfrak{M}_1, x_1\rightsquigarrow_X \mathfrak{M}_2, x_2$ if for any $\varphi\in\mathcal{L}$, $\mathfrak{M}_1, x_1\succ \varphi$ implies $\mathfrak{M}_2, x_2\succ \varphi$.
 
3.~$\mathfrak{M}_1,a_1\leftsquigarrow_A \mathfrak{M}_2,a_2$ if $\mathfrak{M}_2, a_2\rightsquigarrow_A \mathfrak{M}_1, a_1$. $\mathfrak{M}_1, x_1\leftsquigarrow_X \mathfrak{M}_2, x_2$ if $\mathfrak{M}_2, x_2\rightsquigarrow_X \mathfrak{M}_1, x_1$.

4.~$\mathfrak{M}_1, a_1\leftrightsquigarrow_A \mathfrak{M}_2, a_2$ if $\mathfrak{M}_1, a_1\rightsquigarrow_A \mathfrak{M}_2, a_2$ and $\mathfrak{M}_1,a_1\leftsquigarrow_A \mathfrak{M}_2,a_2$. 
    
5.~$\mathfrak{M}_1, x_1\leftrightsquigarrow_X \mathfrak{M}_2, x_2$ if $\mathfrak{M}_1, x_1\rightsquigarrow_X \mathfrak{M}_2, x_2$ and $\mathfrak{M}_1, x_1\leftsquigarrow_X \mathfrak{M}_2, x_2$.

%$a_1$ and $a_2$ are {\em $A$-equivalent}
%$x_1$ and $x_2$ are {\em $X$-equivalent}
\end{definition}
%Given two $\mathrm{LE}$-models $\mathfrak{M}_1$ and $\mathfrak{M}_2$, we use $\rightsquigarrow_A$, $\leftsquigarrow_A$,  $\leftrightsquigarrow_A$, $\rightsquigarrow_X$, $\leftsquigarrow_X$ and $\leftrightsquigarrow_X$ to denote the binary relations between $\mathfrak{M}_1$ and $\mathfrak{M}_2$ which are defined in Definition \ref{def: modally equivalent}. 

For ease of notation, we will use $a_1\rightsquigarrow_A a_2$ instead of $\mathfrak{M}_1, a_1\rightsquigarrow_A \mathfrak{M}_2, a_2$ when $\mathfrak{M}_1$ and $\mathfrak{M}_2$ are clear from the context.  We use similar shorthands for the other relations  defined in Definition \ref{def: modally equivalent}. The following definition generalizes simulations and  bisimulations on Kripke models to $\mathrm{LE}$-models. 
\begin{definition}\label{def:simulation and bisimulation}
Let $\mathfrak{M}_1=(\mathfrak{F}_1, V_1)$ and $\mathfrak{M}_2=(\mathfrak{F}_2, V_2)$ be any $\mathrm{LE}$-models where $\mathfrak{F}_1=(A_1,X_1,I_1,{R_\Box}_1, {R_{\Diamond}}_1)$ and $\mathfrak{F}_2=(A_2,X_2,I_2,{R_\Box}_2, {R_{\Diamond}}_2)$. A {\em simulation} from $\mathfrak{M}_1$ to $\mathfrak{M}_2$ is a pair of relations $(S,T)$ such that $S \subseteq A_1 \times A_2$ and $T \subseteq X_1 \times X_2$ satisfying the following conditions: for any $a_1\in A_1$, $x_1\in X_1$, $a_2\in A_2$ and $x_2\in X_2$, 

1.~If $a_1 S a_2$, then $\forall p\in\atprop$, if $ a_1 \in \val{V_1 (p) }$, then   $ a_2 \in \val{V_2 (p)}$.

2.~If $x_1 T x_2$, then $\forall p\in\atprop$, if $ x_2 \in \descr{V_2 (p) }$, then   $ x_1 \in \descr{V_1 (p) }$.

3.~If $a_1 S a_2$ and $a_2 I^c_2 x_2$, then $\exists x_1 \in X_1$ such that $a_1 I^c_1 x_1$ and  $x_1 T x_2$.

4.~If $x_1 T x_2$ and $a_1 I^c_1 x_1$, then $\exists a_2 \in A_2$ such that $a_2 I^c_2 x_2$ and  $a_1 S a_2$.

5.~If $a_1 S a_2$ and $a_2 R^c_{\Box_2} x_2$, then $\exists x_1 \in X_1$ such that  $a_1 R^c_{\Box_1} x_1$ and  $x_1 T x_2$.

6.~If $x_1 T x_2$ and $x_1 R^c_{\Diamond_1} a_1$, then $\exists a_2 \in A_2$ such that  $x_2 R^c_{\Diamond_2} a_2$ and  $a_1 S a_2$.

We write $\mathfrak{M}_1, a_1\rightrightarrows\mathfrak{M}_2, a_2$ (resp.~$\mathfrak{M}_1, x_1\rightrightarrows\mathfrak{M}_2, x_2$) if there exists a simulation $(S, T)$ from $\mathfrak{M}_1$ to $\mathfrak{M}_2$ such that $a_1Sa_2$ (resp.~$x_1Tx_2$), and $\mathfrak{M}_1, a_1\leftleftarrows\mathfrak{M}_2, a_2$ (resp.~$\mathfrak{M}_1, x_1\leftleftarrows\mathfrak{M}_2, x_2$) if $\mathfrak{M}_2, a_2\rightrightarrows\mathfrak{M}_1, a_1$ (resp.~$\mathfrak{M}_2, x_2\rightrightarrows\mathfrak{M}_1, x_1$). A {\em bisimulation} between $\mathfrak{M}_1$ and $\mathfrak{M}_2$ is simulation $(S,T)$ from $\mathfrak{M}_1$ to $\mathfrak{M}_2$, such that $(S^{-1},T^{-1})$ is a simulation from $\mathfrak{M}_2$ to $\mathfrak{M}_1$.
\end{definition}

 This definition naturally generalizes sumulations and bisimulations  on the Kripke models in the following sense. In \cite[Section 3.3]{conradie2021rough}, it was showed that $\mathrm{LE}$-models can be seen as a generalization of Kripke models by associating any Kripke model $\mathcal{M}=(W, R, V)$ with an $\mathrm{LE}$-model $\mathrm{LE}(\mathcal{M})=(W_A, W_X, \neq, R_{\Box}, R_{\Diamond}, V')$, where $W_A=W=W_B$ and $R_{\Box}=R_{\Diamond}\subseteq W\times W$ such that $w_1 R_\Box w_2$ iff $w_1 R^c w_2$ iff $w_1 R_{\Diamond} w_2$, and for any $p\in\atprop$, $V'(p) = (V(p), V(p)^c)$. This construction ensures that for any points in two Kripke models $\mathcal{M}$ and $\mathcal{M}'$  are modally equivalent, their liftings in the $\mathrm{LE}$-models $\mathrm{LE}(\mathcal{M})$ and $\mathrm{LE}(\mathcal{M}')$ are modally equivalent w.r.t.~the language $\mathcal{L}$.  For more details, see \cite{conradie2021rough}. The following lemma shows that any simulation (resp.~bisimulation) on a Kripke model $\mathcal{M}$ naturally induces a simulation (resp.~bisimulation) on $\mathrm{LE}(\mathcal{M})$. 

%In \cite{a}, it was showed that $\mathrm{LE}$-models can be considered as a generalization of Kripke models by associating any Kripke model $\mathcal{M}=(W, R_\Box, R_\Diamond, V)$ with a $\mathrm{LE}$-model  $\mathbb{P}(M)=(W_A, W_X, I_\neq, I_{R_\Box^c},  J_{R_\Diamond^c}, V')$, where $W_A$, $W_X$ are two copies of $W$, $I_\neq \subseteq W_A \times W_x$ is such that $w_1 I_\neq w_2$ iff $w_1 \neq w_2$, $I_{R_\Box^c}  \subseteq W_A \times W_X$ is such that $w_1 I_{R_\Box^c}  w_2$ iff $w_1 R_\Box^c w_2$, and $J_{R_\Diamond^c}  \subseteq W_X \times W_A$ is such that $w_1 J_{R_\Diamond^c} w_2$ iff $w_1 R_\Diamond^c w_2$, and for any proposition $p$, $V'(p) = (V(p), V(p)^c)$ (For more details on this generalization, see  \cite{a}.). The following lemma shows that any bisimulation on a Kripke model  $M=(W,R_\Box, R_\Diamond, V)$ naturally defines a bisimulation on the polarity-base model $\mathbb{P}(M)$ in the sense of Definition \ref{def:simulation and bisimulation}. 

\begin{lemma}\label{lemma: generalize bisimulation on Kripke models}
Let $\mathcal{M}_1=(W_1, R_1, V_1)$ and $\mathcal{M}_2=(W_2, 
R_2, V_2)$ be any Kripke models, $\mathrm{LE}(\mathcal{M}_1)=(W_{A_1}, W_{X_1}, \neq, R_{\Box_1}, R_{\Diamond_1}, V_1')$ and $\mathrm{LE}(\mathcal{M}_2)=(W_{A_2}, W_{X_2}, \neq, R_{\Box_2}, R_{\Diamond_2}, V_2')$ be their corresponding $\mathrm{LE}$-models. A relation $Z \subseteq W_1 \times W_2$ between the Kripke models $\mathcal{M}_1$ and $\mathcal{M}_2$ is a  simulation (resp.~bisimulation) iff the tuple $(Z_A, Z_X)$, where $Z_A=Z= Z_X\subseteq W_{A_1} \times  W_{A_2}=W_{X_1}\times  W_{X_2}$ is a simulation (resp.~bisimulation) between the $\mathrm{LE}$-models $\mathrm{LE}(\mathcal{M}_1)$ and $\mathrm{LE}(\mathcal{M}_2)$.
\end{lemma}

\begin{proof}
The proof follows straight-forwardly from the definitions of simulations and bisimulations on Kripke models and $\mathrm{LE}$-models.
\end{proof}

%\begin{lemma}
%Let $M_1=(W_1,{R_\Box}_1, {R_\Diamond}_1, V_1)$ and $M_2=(W_2,{R_\Box}_2, {R_\Diamond}_2, V_2)$ be any Kripke models and let  $\mathbb{P}(M_1)=({W_1}_A, {W_1}_X, {I_1}_\neq, I_{{R_\Box}_1^c},  J_{{R_\Diamond}_1^c}, V_1')$ and $\mathbb{P}(M_2)=({W_2}_A, {W_2}_X, {I_2}_\neq , I_{{R_\Box}_2^c},  J_{{R_\Diamond}_2^c}, V_2')$ be the polarity-based models defined by them.  Let $Z \subseteq W_1 \times W_2$ be any bisimulation between $M_1$ and $M_2$. Then the tuple $(Z_A, Z_X)$, where $Z_A \subseteq {W_1}_A \times {W_2}_A$, and $Z_X \subseteq {W_1}_X \times {W_2}_X$ are such that $w_1 Z_A w_2$ iff $w_1 Z_x w_2$ iff $w_1 Z w_2$ defined a bisimulation between polarity-based models  $\mathbb{P}(M_1)$ and $\mathbb{P}(M_2)$  as per Defitnion \ref{def:simulation and bisimulation}
%\end{lemma}

The following theorem shows that the $\mathcal{L}$-formulas are preserved and reflected by simulations w.r.t.~modal satisfaction relations $\Vdash$ and $\succ$, respectively.

\begin{theorem}\label{thm:simulation invariance}
Let $\mathfrak{M}_1=(\mathfrak{F}_1, V_1)$ and $\mathfrak{M}_2=(\mathfrak{F}_2, V_2)$ be any $\mathrm{LE}$-models, $a_1, x_1$ in $\mathfrak{M}_1$ and $a_2, x_2$ in $\mathfrak{M}_2$, then

%where $\mathfrak{F}_1=(A_1,X_1,I_1,{R_\Box}_1, {R_{\Diamond}}_1)$ and $\mathfrak{F}_2=(A_2,X_2,I_2,{R_\Box}_2, {R_{\Diamond}}_2)$, $a_1\in A_1$

%Let $(S,T)$ be a simulation from $\mathfrak{M}_1$ to $\mathfrak{M}_2$, then for any $\varphi\in\mathcal{L}$,
1.~If $\mathfrak{M}_1, a_1\rightrightarrows\mathfrak{M}_2, a_2$, then $\mathfrak{M}_1, a_1\rightsquigarrow_A\mathfrak{M}_2, a_2$. %For any $a_1\in A_1$ and $a_2\in A_2$, $a_1 S a_2$ implies $a_1\rightsquigarrow_A a_2$.

2.~If $\mathfrak{M}_1, x_1\rightrightarrows\mathfrak{M}_2, x_2$, then $\mathfrak{M}_1, x_1\leftsquigarrow_X\mathfrak{M}_2, x_2$. %For any $x_1\in X_1$ and $x_2\in X_2$, $x_1 T x_2$ implies $x_2\rightsquigarrow_X x_1$.

\end{theorem}

\begin{proof}
    See Appendix \ref{Proof of thm:simulation invariance}.
\end{proof}

As a corollary, we get that the $\mathcal{L}$-formulas are bisimulation-invariant. 

\begin{corollary}
    
\label{thm:bisimulation invariance for the first definition}
Let $\mathfrak{M}_1=(\mathfrak{F}_1, V_1)$ and $\mathfrak{M}_2=(\mathfrak{F}_2, V_2)$ be any $\mathrm{LE}$-models, where $\mathfrak{F}_1=(A_1,X_1,I_1,{R_\Box}_1, {R_{\Diamond}}_1)$ and $\mathfrak{F}_2=(A_2,X_2,I_2,{R_\Box}_2, {R_{\Diamond}}_2)$, and $(S,T)$ be a bisimulation between them. Then 

1.~For any $a_1\in A_1$ and $a_2\in A_2$, $a_1 S a_2$ implies $\mathfrak{M}_1, a_1\leftrightsquigarrow_A \mathfrak{M}_2, a_2$.

2.~For any $x_1\in X_1$ and $x_2\in X_2$, $x_1 T x_2$ implies $ \mathfrak{M}_1, x_1\leftrightsquigarrow_X \mathfrak{M}_2, x_2$.
\end{corollary}
\begin{proof}
   %See Appendix \ref{Proof of Theorem bisimulation invariance for the first definition}.
Note that by definition for  any  bisimulation $(S,T)$,  both $(S,T)$ and $(S^{-1}, T^{-1})$ are simulations. The  proof follows immediately  by Theorem \ref{thm:simulation invariance}. 
\end{proof}

We now try to prove the Hennessy-Milner theorem for the polarity-based semantics of non-distributive modal logic. We begin by generalizing the notion of image-finite models from Kripke semantics to polarity-based semantics.

\begin{definition}\label{def:image finite}
    An $\mathrm{LE}$-model $\mathfrak{M}=(A, X, I, R_{\Box}, R_{\Diamond}, V)$ is {\em image-finite} if for any $a\in A$ and $x\in X$, the sets $\{a'\mid a'I^{c}x\}, \{x'\mid aI^{c} x'\}, \{a'\mid a'R^{c}_{\Box}x\}$ and $\{x'\mid x'R^{c}_{\Diamond} a\}$ are all finite.
\end{definition}
Note that a Kripke model $\mathcal{M}$  is image-finite iff  $\mathrm{LE}(\mathcal{M})$  is image-finite. Thus, this definition can be seen as a generalization of image-finite models to polarity-based setting.  However, as the language of non-distributive modal logic ($\mathcal{L}$) does not contain negation, the proof strategy for the classical Hennessy-Milner theorem  cannot be adopted in the non-distributive setting. In fact, as  the following  example shows,  Hennessy-Milner theorem does not hold for the polarity-based semantics of non-distributive modal logic for the notion of bisimulation given by Definition \ref{def:simulation and bisimulation}.

\iffalse
\begin{example} %\marginnote{Km:needs figure}
Consider two $\mathrm{LE}$-models $\mathfrak{M}_i = (A_i, X_i, I_i, {R_\Box}_i, {R_\Diamond}_i, V_i)$ for $i\in \{1,2\}$ such that $A_1=\{a_1, b_1\}$, $X_1 =\{x_1, y_1\}$, $I_1= \{b_1, x_1\}$, $V_1(p)=(\emptyset, X_1)$ and $V_1(q) = (\{b_1\}, \{x_1\})$, $A_2=\{a_2\}$, $X_2=\{x_2\}$, and $I_2= {R_\Box}_i= {R_\Diamond}_i=\emptyset$ for $i\in \{1,2\}$. 
Then, note that $a_1 \leftrightsquigarrow_A a_2$ , but they can not have any bisimulation between them as $a_1 I_1^c y_1$, and $\mathfrak{M}_1, y_1  \not \succ q$, but for  only element $y_2 \in Y_2$, we have $\mathfrak{M}_2, y_2 \succ q$. 
\end{example}
\fi

\begin{example}
Let $\mathfrak{M}_i = (A_i, X_i, I_i, {R_\Box}_i, {R_\Diamond}_i, V_i)$ for $i\in \{1,2\}$ be two image-finite models
    %Consider two image-finite $\mathrm{LE}$-models $\mathfrak{M}_i = (A_i, X_i, I_i, {R_\Box}_i, {R_\Diamond}_i, V_i)$ where $i\in \{1,2\}$ 
    such that $A_1=\{a_1, b_1\}$, $X_1 =\{x_1, y_1\}$, $I_1= \{(b_1, x_1)\}$, $A_2=\{a_2\}$, $X_2=\{x_2\}$ and $I_2= {R_\Box}_i= {R_\Diamond}_i=\emptyset$ for $i\in \{1,2\}$. Let $p$, $q$ be  any two fixed propositions. Suppose $V_1$ and $V_2$ are  such that, $V_1(q)=(\{b_1\}, \{x_1\})$, $p\in\atprop$, $V_1(p)=(\{a_1, b_1\}, \emptyset)$, $V_2(q)=(\emptyset,\{x_2\})$ and $V_2(p)=(\{a_2\},\emptyset)$. These models are depicted in Figure \ref{fig:counter-example}. It is  easy to check that $a_1\leftrightsquigarrow_A a_2$.

    Suppose there exists a bisimulation $(S, T)$ between $\mathfrak{M}_1$ and $\mathfrak{M}_2$, such that, $a_1Sa_2$. Then, by item 3 of Definition \ref{def:simulation and bisimulation} for simulation $(S^{-1}, T^{-1})$
    and $a_1I^c_1y_1$, we must have   $y_1Tx_2$. However, $\mathfrak{M}_1, y_1\nsucc q$ and $\mathfrak{M}_2, x_2\succ q$, which contradicts item 2 of Definition \ref{def:simulation and bisimulation} for simulation $(S,T)$. Hence, there can not be a bisimulation between  $a_1$ and $a_2$.
\end{example}

\begin{figure}
    \centering
        \includegraphics[scale=0.23]{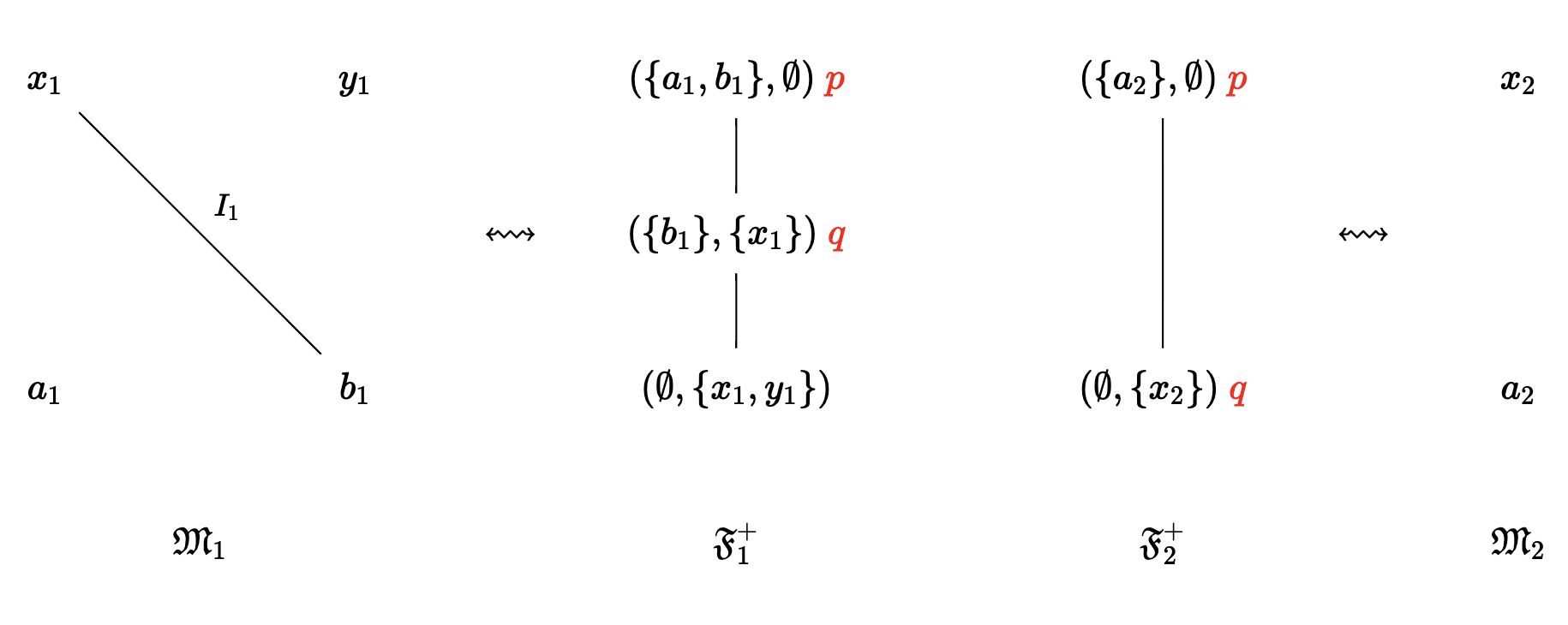}
    \caption{Models $\mathfrak{M}_1$ and  $\mathfrak{M}_2$ and their complex algebras}
    \label{fig:counter-example}
\end{figure}

Therefore, to obtain the Hennessy-Milner theorem for polarity-based semantics, we introduce the notion of {\em bisimilarity}, which is founded on the definition of simulations between $\mathrm{LE}$-models. This notion naturally generalizes the definition of bisimulation outlined in Definition \ref{def:simulation and bisimulation}.

\begin{definition} \label{def:bisimilar}
Let $\mathfrak{M}_1=(\mathfrak{F}_1, V_1)$ and $\mathfrak{M}_2=(\mathfrak{F}_2, V_2)$ be any $\mathrm{LE}$-models, where $\mathfrak{F}_1=(A_1,X_1,I_1,{R_\Box}_1, {R_{\Diamond}}_1)$, $\mathfrak{F}_2=(A_2,X_2,I_2,{R_\Box}_2, {R_{\Diamond}}_2)$. For any $a_1, x_1$ in $\mathfrak{M}_1$ and $a_2, x_2$ in $\mathfrak{M}_2$, we say that $a_1$ and $a_2$ (resp.~$x_1$ and $x_2$) are {\em bisimilar}, and denote it by $\mathfrak{M}_1, a_1\rightleftarrows\mathfrak{M}_2, a_2$ (resp.~$\mathfrak{M}_1, x_1\rightleftarrows\mathfrak{M}_2, x_2$), if both $\mathfrak{M}_1, a_1\rightrightarrows\mathfrak{M}_2, a_2$ and $\mathfrak{M}_1, a_1\leftleftarrows\mathfrak{M}_2, a_2$ (resp.~$\mathfrak{M}_1, x_1\rightrightarrows\mathfrak{M}_2, x_2$ and $\mathfrak{M}_1, x_1\leftleftarrows\mathfrak{M}_2, x_2$).

\iffalse
$\mathfrak{M}_1, a_1\rightrightarrows\mathfrak{M}_2, a_2$ (resp.~$\mathfrak{M}_1, x_1\rightrightarrows\mathfrak{M}_2, x_2$) if there exists a simulation $(S, T)$ from $\mathfrak{M}_1$ to $\mathfrak{M}_2$ such that $a_1Sa_2$ (resp.~$x_1Tx_2$). $\mathfrak{M}_1, a_1\leftleftarrows\mathfrak{M}_2, a_2$ (resp.~$\mathfrak{M}_1, x_1\leftleftarrows\mathfrak{M}_2, x_2$) if $\mathfrak{M}_2, a_2\rightrightarrows\mathfrak{M}_1, a_1$ (resp.~$\mathfrak{M}_2, x_2\rightrightarrows\mathfrak{M}_1, x_1$). $a_1$ and $a_2$ (resp.~$x_1$ and $x_2$) are {\em bisimilar}, which is denoted as $\mathfrak{M}_1, a_1\rightleftarrows\mathfrak{M}_2, a_2$ (resp.~$\mathfrak{M}_1, x_1\rightleftarrows\mathfrak{M}_2, x_2$), if both $\mathfrak{M}_1, a_1\rightrightarrows\mathfrak{M}_2, a_2$ and $\mathfrak{M}_1, a_1\leftleftarrows\mathfrak{M}_2, a_2$ (resp.~ $\mathfrak{M}_1, x_1\rightrightarrows\mathfrak{M}_2, x_2$ and $\mathfrak{M}_1, x_1\leftleftarrows\mathfrak{M}_2, x_2$).
\fi

%there exists a simulation $Z_1=(S_1, T_1)$ from $\mathfrak{M}_1$ to $\mathfrak{M}_2$ and a simulation $Z_2=(S_2, T_2)$ from $\mathfrak{M}_2$ to $\mathfrak{M}_1$ such that $a_1 S_1 a_2$ and $a_2 S_2 a_1$ (resp.~$x_1 T_1 x_2$ and $x_2 T_2 x_1$).
\end{definition}
Henceforth, when we say that two elements are bisimilar, we refer to Definition \ref{def:bisimilar}, unless stated otherwise. We now show that even though Hennessy-Milner theorem for bisimulations does not generalize to the polarity-based setting,  its counterpart for simulations does so.

\begin{theorem}[Hennesy-Milner Theorem]\label{thm:finite image implies bisimilar}
Let $\mathfrak{M}_1=(\mathfrak{F}_1, V_1)$ and $\mathfrak{M}_2=(\mathfrak{F}_2, V_2)$ be any image-finite $\mathrm{LE}$-models, $a_1, x_1$ in $\mathfrak{M}_1$ and $a_2, x_2$ in $\mathfrak{M}_2$, then

%where $\mathfrak{F}_1=(A_1,X_1,I_1,{R_\Box}_1, {R_{\Diamond}}_1)$ and $\mathfrak{F}_2=(A_2,X_2,I_2,{R_\Box}_2, {R_{\Diamond}}_2)$. Let $a_1\in A_1$, $x_1\in X_1$, $a_2\in A_2$ and $x_2\in X_2$. 
1.~$\mathfrak{M}_1, a_1\rightsquigarrow_A \mathfrak{M}_2, a_2$ if and only if $\mathfrak{M}_1, a_1\rightrightarrows\mathfrak{M}_2, a_2$. %there exists a simulation $(S,T)$ from $\mathfrak{M}_1$ to $\mathfrak{M}_2$ such that $a_1 S a_2$.

2.~$\mathfrak{M}_1, x_1\rightsquigarrow_X \mathfrak{M}_2, x_2$ if and only if $\mathfrak{M}_1, x_1\leftleftarrows\mathfrak{M}_2, x_2$. %$x_1 \rightsquigarrow_X x_2$ iff there exists a simulation $(S,T)$ from $\mathfrak{M}_2$ to $\mathfrak{M}_1$ such that $x_2 T x_1$.

\end{theorem}
\begin{proof}
The right to left implication for both items follows from Theorem \ref{thm:simulation invariance}. 
We prove the left to right implication only for the first item, as the proof for the second is analogous.

For the left to right implication of the first item, we claim that $(\rightsquigarrow_A,\leftsquigarrow_X)$ is a simulation from $\mathfrak{M}_{1}$ to $\mathfrak{M}_{2}$, where $(a, a')\in\rightsquigarrow_A$ (denoted as $a\rightsquigarrow_A a'$) iff $\mathfrak{M}_1, a\rightsquigarrow_A \mathfrak{M}_2, a'$ and $(x, x')\in\leftsquigarrow_X$ (denoted as $x\leftsquigarrow_X x'$) iff $\mathfrak{M}_1, x\leftsquigarrow_X \mathfrak{M}_2, x'$. Assume $a_{1}\rightsquigarrow_A a_{2}$ and $x_{1}\leftsquigarrow_X x_{2}$, the items 1 and 2 of Definition \ref{def:simulation and bisimulation} are satisfied immediately.

For item 3, assume $a_{1} \rightsquigarrow_A a_{2}$ and $a_{2}I^{c}_{2}x_{2}$. Therefore, $\mathfrak{M}_2, a_2\nVdash\bot$. By definition of $\rightsquigarrow_A $, we have $\mathfrak{M}_{1}, a_{1}\nVdash\bot$, which means that there exists $x_{1}$ such that $a_{1}I^{c}_{1}x_{1}$. Therefore, the set $Y':=\{x\mid a_{1}I^{c}_{1}x\}$ is a non-empty finite (by image-finiteness) set . Let $Y'$ be the set $\{x_{1}',\cdots ,x'_{n}\}$. Assume for any $x\in Y'$, it is not the case that $x\leftsquigarrow_X x_{2}$. So, there exist finitely many formulas $\psi_{1}, \cdots ,\psi_{n}$ such that for each $i$, $\mathfrak{M}_{2},x_2\succ\psi_{i}$ and $\mathfrak{M}_{1}, x'_{i}\nsucc\psi_{i}$.  As $a_2I^c_2 x_2$, there is $\mathfrak{M}_2, a_2\nVdash\psi_{1}\vee\cdots\vee\psi_{n}$. On the other hand, let $x'$ be any element in $ X_1$. Thus, $\mathfrak{M}_{1}, x'\succ\psi_{1}\vee \cdots \vee\psi_{n}$ implies $M_{1}, x'\succ\psi_{i}$ for each $i$, which means that $x'\notin Y'$, so $a_{1}I_{1}x'$. Hence, there is $\mathfrak{M}_{1}, a_{1}\Vdash\psi_{1}\vee...\vee\psi_{n}$, which contradicts $a_{1}\rightsquigarrow_A a_{2}$. So, there must exist $x_{1}$ such that $a_{1}I^{c}_{1}x_{1}$ and $x_{1}\leftsquigarrow_X x_{2}$. 

Item  4 can be proved similarly to item 3.

%For condition 4, assume $x_{1} \leftsquigarrow_X x_{2}$ and $a_{1}I^{c}_{1}x_{1}$. Therefore, $\mathfrak{M}, x_1\nVdash\top$. By definition of $\leftsquigarrow_X$, $\mathfrak{M}, x_2\nVdash\top$, which means that there exists $a_{2}$ such that $a_{2}I^{c}_{2}x_{2}$. Therefore, the set $B':=\{a\mid aI^{c}_{2}x_{2}\}$ is a non-empty finite (by image-finiteness) set. Let  $B'$ be the set $\{a'_{1},...,a'_{n}\}$. Assume for any $a'\in B'$, it is not the case that $a_1\rightsquigarrow_A a'$. Therefore, there exist finitely many formulas $\psi_{1},\cdots ,\psi_{n}$ such that for each $i$, $\mathfrak{M}_{1}, a_{1}\Vdash\psi_{i}$ and $\mathfrak{M}_{2}, a'_{i}\nsucc\psi_{i}$. Therefore, there is $\mathfrak{M}_{1}, a_{1}\Vdash\psi_{1}\wedge \cdots \wedge\psi_{n}$, which implies $\mathfrak{M}_{1}, x_{1}\nsucc\psi_{1}\wedge \cdots \wedge\psi_{n}$ because of $a_1I^c_1x_1$. On the other hand, let $b'$ be any element in $A_2$. Therefore, $\mathfrak{M}_{2}, b'\Vdash\psi_{1}\wedge\cdots\wedge\psi_{n}$ implies that $\mathfrak{M}_{2}, b'\Vdash\psi_{i}$ for each $i$, which means that $b'\notin B'$, so $b' I_{2}x_2$. Therefore, there is $\mathfrak{M}_{2}, x_{2}\succ\psi_{1}\wedge \cdots \wedge\psi_{n}$, which contradicts $x_{1}\leftsquigarrow_{X} x_{2}$. Therefore, there must exist $a_{2}$ such that $a_{2}I^{c}_{1}x_{2}$ and $a_{1}\rightsquigarrow_{A} a_{2}$.

For item 5, assume $a_{1} \rightsquigarrow_A a_{2}$ and $a_{2}R^{c}_{\Box_{2}}x_{2}$. Therefore, $\mathfrak{M}_{2}, a_{2}\nVdash\Box\bot$, which implies $\mathfrak{M}_{1}, a_{1}\nVdash\Box\bot$. Therefore, there  exists $x_{1}\in X_1$ such that $a_{1}R^{c}_{\Box_{1}}x_{1}$. Thus, the set $Y':=\{x\mid a_{1}R^{c}_{\Box_{1}}x\}$ is a non-empty finite (by image-finiteness) set. Let $Y'$ be the set  $\{x'_{1},\cdots,x'_{n}\}$. Assume for any $x'\in Y'$, it is not the case that $x'\leftsquigarrow_X x_{2}$. Therefore, there exist finitely many formulas $\psi_{1},\cdots,\psi_{n}$ such that for each $i$, $\mathfrak{M}_{2}, x_{2}\succ\psi_{i}$ and $\mathfrak{M}_{1}, x'_{i}\nsucc\psi_{i}$. Therefore,  $\mathfrak{M}_{2}, x_{2}\succ\psi_{1}\vee \cdots \vee\psi_{n}$, which implies $\mathfrak{M}_{2}, a_{2}\nVdash\Box(\psi_{1}\vee \cdots \vee\psi_{n})$. On the other hand, let $y'$ be any element in $X_1$. $\mathfrak{M}_{1}, y'\succ\psi_{1}\vee \cdots \vee\psi_{n}$ implies $\mathfrak{M}_{1}, y'\Vdash\psi_{i}$ for each $i$, which means $y'\notin Y'$, so $a_{1}R_{\Box_{1}}y'$. Therefore, $\mathfrak{M}_{1}, a_{1}\Vdash \Box(\psi_{1}\vee \cdots \vee\psi_{n})$, which contradicts $a_{1}\rightsquigarrow_A a_{2}$. So, there must exist $x_{1}$ such that $a_{1}R^{c}_{\Box_{1}}x_{1}$ and $x_{1}\leftsquigarrow_X x_{2}$. 

Item 6 can be proved similarly to item 5.

%For condition 6, assume $x_{1} \leftsquigarrow_X  x_{2}$ and $x_{1}R^{c}_{\Diamond_{1}}a_{1}$. Therefore, there is $\mathfrak{M}_{1}, x_{1}\nsucc\Diamond\top$, which implies $\mathfrak{M}_{2}, x_{2}\nsucc\Diamond\top$. Therefore, there exists $a_{2}$ such that $x_{2}R^{c}_{\Diamond_{2}}a_{2}$. Thus, the set $B':=\{a\mid x_{2}R^{c}_{\Diamond_{2}}a\}$ is a non-empty finite (by image-finiteness) set. Let $B'$ be the set $\{a'_{1},...,a'_{n}\}$. Assume for any $a'\in B'$, it is not the case that $a_1\rightsquigarrow_A a'$. Therefore, there exist finitely many formulas $\psi_{1},\cdots,\psi_{n}$ such that for each $i$, $\mathfrak{M}_{1}, a_{1}\Vdash\psi_{i}$ and $\mathfrak{M}_{2}, a'_{i}\nVdash\psi_{i}$. Therefore, there is $\mathfrak{M}_{1}, a_{1}\Vdash\psi_{1}\wedge \cdots \wedge\psi_{n}$, which implies $\mathfrak{M}_{1}, x_{1}\nsucc\psi_{1}\Diamond(\psi_{1}\wedge...\wedge\psi_{n})$ because of $x_{1}R^{c}_{\Diamond_{1}}a_{1}$. On the other hand, let $b'$ be any element in $A_2$. $\mathfrak{M}_{2}, b'\Vdash\psi_{1}\wedge \cdots \wedge\psi_{n}$ implies $\mathfrak{M}_{2}, b'\Vdash\psi_{i}$ for each $i$, which means that $b'\notin B'$, so $x_{2}R_{\Diamond_{2}}b'$. Therefore, there is $\mathfrak{M}_{2}, x_{2}\succ\Diamond(\psi_{1}\wedge \cdots \wedge\psi_{n})$ which contradicts $x_{1}\leftsquigarrow_X x_{2}$. Therefore, there must exist $a_{2}$ such that $x_{2}R^{c}_{\Diamond_{2}}a_{2}$ and $a_{1}\rightsquigarrow_A a_{2}$. 
\end{proof}
The above theorem can be seen as a generalization of Henessey-Milner theorem to non-distributive setting as it relates invariance under formulas to simulations. The following corollary is immediate from the above theorem.
%Let $a_1 \leftrightsquigarrow_A a_2$ iff $a_1 \rightsquigarrow_A a_2$ and  $a_2 \rightsquigarrow_A a_1$. Let $x_1 \leftrightsquigarrow_X x_2$ iff $x_1 \rightsquigarrow_X x_2$ and  $x_2 \rightsquigarrow_X x_1$. 

\begin{corollary}\label{Cor:Hennesy-Milner}
Let $\mathfrak{M}_1=(\mathfrak{F}_1, V_1)$ and $\mathfrak{M}_2=(\mathfrak{F}_2, V_2)$ be any image-finite $\mathrm{LE}$-models, $a_1, x_1$ in $\mathfrak{M}_1$ and $a_2, x_2$ in $\mathfrak{M}_2$, then
 %Let $\mathfrak{M}_1=(\mathfrak{F}_1, V_1)$ and $\mathfrak{M}_2=(\mathfrak{F}_2, V_2)$ be any image-finite $\mathrm{LE}$-models, where $\mathfrak{F}_1=(A_1,X_1,I_1,{R_\Box}_1, {R_{\Diamond}}_1)$ and $\mathfrak{F}_2=(A_2,X_2,I_2,{R_\Box}_2, {R_{\Diamond}}_2)$, then
 
1.~$\mathfrak{M}_1, a_1 \leftrightsquigarrow_A \mathfrak{M}_2, a_2$ if and only if $\mathfrak{M}_1, a_1\rightleftarrows\mathfrak{M}_2, a_2$.

2.~$\mathfrak{M}_1, x_1 \leftrightsquigarrow_X \mathfrak{M}_2, x_2$ if and only if $\mathfrak{M}_1, x_1\rightleftarrows\mathfrak{M}_2, x_2$.

\end{corollary}

As the simulations on $\mathrm{LE}$-models can be seen as a generalizations of simulations on Kripke models, the converse of Theorem \ref{thm:simulation invariance} does not always hold. That is, image-finiteness condition in  Theorem \ref{thm:finite image implies bisimilar} can not be dropped. 
\begin{figure}
    \centering
        \includegraphics[scale=0.1]{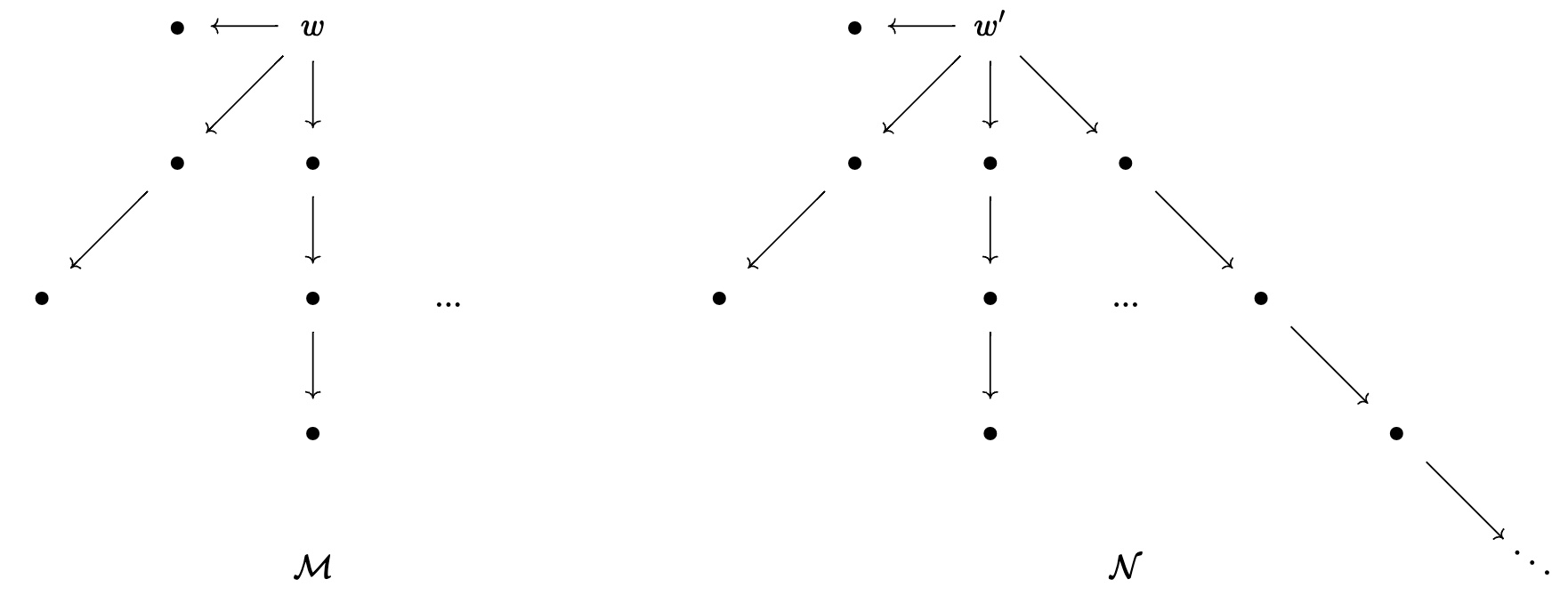}
    \caption{Kripke models $\mathcal{M}$ and $\mathcal{N}$}
    \label{fig:counter-example 2}
\end{figure}

\begin{example}
Consider two Kripke models $\mathcal{M}$ and $\mathcal{N}$ in Fig \ref{fig:counter-example 2}, where any $p\in\atprop$ is false at all the  points. It is easy to check that $w$ and $w'$ are modally equivalent in $\mathcal{M}$ and $\mathcal{N}$. Therefore ${w'}_A\rightsquigarrow_A w_A$ on their corresponding $\mathrm{LE}$-models $\mathrm{LE}(\mathcal{N})$ and $\mathrm{LE}(\mathcal{M})$. Assume that there exists a simulation $(S, T)$ from $\mathrm{LE}(\mathcal{N})$ to $\mathrm{LE}(\mathcal{M})$ such that ${w'}_ASw_A$. Then, $S$ must be a simulation from $\mathcal{N}$ to $\mathcal{M}$. However, it is not possible to have simulations from $\mathcal{N}$ to $\mathcal{M}$, so there is a contradiction. Therefore, there is no simulation from $\mathrm{LE}(\mathcal{N})$ to $\mathrm{LE}(\mathcal{M})$ which can link ${w'}_A$ and $w_A$.  This  shows that converse of item 1 in Theorem \ref{thm:finite image implies bisimilar} is not true in general. We can similarly show that converse of item 2 in Theorem \ref{thm:finite image implies bisimilar} need not hold in general. 
\end{example}

\section{Modal Saturation via Filter-Ideal Extensions}\label{sec:Modal Saturation via Filter-Ideal Extensions}

In this section, we continue to investigate the relation between modal equivalence and bisimilarity. We first introduce the notion of modally-saturated ($\mathrm{M}$-saturated) $\mathrm{LE}$-models and the Hennessy-Milner property. Then, we show that the class of $\mathrm{M}$-saturated $\mathrm{LE}$-models has the Hennessy-Milner property. After that, we introduce the notion of filter-ideal extension of $\mathrm{LE}$-models and show that any points in $\mathrm{LE}$-models are modally equivalent if and only if they are bisimilar on their filter-ideal extensions.

\subsection{Modally-saturated \texorpdfstring{$\mathrm{LE}$}{LE}-models and Hennessy-Milner Property}\label{sec:Modal Saturation via Ultrafilter Extensions}
\begin{definition}\label{def: HMP}
    Let $\mathrm{K}$ be a class of $\mathrm{LE}$-models. $\mathrm{K}$ is a {\em Hennessy-Milner class} or has the {\em Hennessy-Milner property} if for any $\mathrm{LE}$-models $\mathfrak{M}_1, \mathfrak{M}_2\in\mathrm{K}$, and $a_1$, $x_1$ in $\mathfrak{M}_1$, and $a_2$, $x_2$ in $\mathfrak{M}_2$,

1.~$\mathfrak{M}_1, a_1\rightsquigarrow_X \mathfrak{M}_2, a_2$ implies $\mathfrak{M}_1, a_1\rightrightarrows\mathfrak{M}_2, a_2$, and

2.~$\mathfrak{M}_1, x_1\rightsquigarrow_A \mathfrak{M}_2, x_2$ implies $\mathfrak{M}_1, x_1\leftleftarrows\mathfrak{M}_2, x_2$.

\end{definition}

Theorem \ref{thm:finite image implies bisimilar} shows that image-finite $\mathrm{LE}$-models have the Hennessy-Milner property; we now introduce a larger class of models which retains the property.

\begin{definition}\label{def:M-saturation}
    Let $\mathfrak{M}=(A, X, I, {R_\Box}, {R_{\Diamond}}, V)$ be an $\mathrm{LE}$-model, $A'\subseteq A$ and $X'\subseteq X$. Let $\Sigma$ be a set of $\mathcal{L}$-formulas. $\Sigma$ is {\em satisfiable} in $A'$ (resp.\ $X'$) if for any $\varphi\in\Sigma$, there exists a point $a$ (resp.\ $x$) such that, $\mathfrak{M}, a\Vdash\varphi$ (resp.\ $\mathfrak{M}, x\succ\varphi$). $\Sigma$ is {\em finitely satisfiable} in $A'$ (resp.\ $X'$) if any finite subset of $\Sigma$ is satisfiable in $A'$ (resp.\ $X'$).  $\mathfrak{M}$ is {\em modally-saturated} ($\mathrm{M}$-saturated for short) if for any set of formulas $\Sigma$, $a\in A$ and $x\in X$ the following conditions hold:
    %I know, I know, this is fucking horrendous, but it works
    
\!\!\!\!\!\!\!1.~If $\Sigma$ is finitely satisfiable in $\{x'\mid aI^c x'\}$, then $\Sigma$ is satisfiable in  $\{x'\mid aI^c x'\}$.

\!\!\!\!\!\!\!2.~If $\Sigma$ is finitely satisfiable in $\{a'\mid a'I^c x\}$, then $\Sigma$ is satisfiable in $\{a'\mid a'I^c x\}$.

\!\!\!\!\!\!\!3.~If $\Sigma$ is finitely satisfiable in $\{x'\mid aR^c_\Box x'\}$, then $\Sigma$ is satisfiable in  $\{x'\mid aR^c_\Box x'\}$.

\!\!\!\!\!\!\!4.~If $\Sigma$ is finitely satisfiable in $\{a'\mid xR^c_\Diamond a'\}$, then $\Sigma$ is satisfiable in $\{a'\mid xR^c_\Diamond a'\}$.
\end{definition}

\begin{theorem}\label{Thm:M-saturated is HM}
    The class of $\mathrm{M}$-saturated models has Hennessy-Milner property.
\end{theorem}

\begin{proof}
    See Appendix \ref{proof of M-saturated is HM}.
\end{proof}

\subsection{Filter-Ideal Extension of \texorpdfstring{$\mathrm{LE}$}{LE}-Models}
Let $\mathbf{A}=(\mathbf{L}, \Box, \Diamond)$ be any $\mathrm{LE}$-algebra (cf.~\cite[Definition 1]{conradie2018goldblatt}). $\mathbf{A}_+=(\mathsf{Fi}_\mathbf{A},\mathsf{Id}_\mathbf{A},I,R_\Box, R_\Diamond)$ is the {\em filter-ideal frame} of $\mathbf{A}$ (c.f.~\cite[Definition 29]{conradie2018goldblatt}), where $\mathsf{Fi}_\mathbf{A}$ (resp.~$\mathsf{Id}_\mathbf{A}$) is the set of all filters (resp.~ideals) of $\mathbf{L}$ such that for any $F \in \mathsf{Fi}_\mathbf{A}$ and $J \in \mathsf{Id}_\mathbf{A}$,

    %\item[1.] $\mathsf{Fi}_\mathbf{A}$ is the set of filters of $\mathbf{L}$,
    %\item[2.] $\mathsf{Id}_\mathbf{A}$ is the set of ideals of $\mathbf{L}$,
1.~$F I J$ if and only if $F \cap J \neq \emptyset$.

2.~$F R_\Box J$ if and only if there exists $a\in J$ such that $\Box a \in F$.

3.~$J R_\Diamond F$ if and only if there exists $a \in F$ such that $\Diamond a \in J$.

\begin{definition}
Let $\mathfrak{M}=(\mathfrak{F},V)$ be any $\mathrm{LE}$-model with $\mathfrak{F}=(A,X,I, R_\Box, R_\Diamond)$. The {\em filter-ideal extension} of $\mathfrak{M}$ is the tuple $\mathfrak{M}^{\mathsf{FI}}=(\mathfrak{F}^{\mathsf{FI}}, V^{\mathsf{FI}})$, where $\mathfrak{F}^{\mathsf{FI}} =(\mathfrak{F}^+)_+ $ and for any $p\in\atprop$, $V^{\mathsf{FI}}(p)=(\val{V^{\mathsf{FI}}(p)},\descr{V^{\mathsf{FI}}(p)})=(\{F \mid (\val{p},\descr{p}) \in F\}, \{J \mid (\val{p},\descr{p}) \in J\})$.
\end{definition}

\begin{remark}
In order to check that $\mathfrak{M}^{\mathsf{FI}}$ is well defined we need to check that for any $p\in\atprop$,  $(\val{V^{\mathsf{FI}}(p)},\descr{V^{\mathsf{FI}}(p)})$ indeed forms a concept of $\mathfrak{F}$. We only prove  $\val{V^{\mathsf{FI}}(p)}^\uparrow=\descr{V^{\mathsf{FI}}(p)}$ here. The proof for  $\descr{V^{\mathsf{FI}}(p)}^\downarrow=\val{V^{\mathsf{FI}}(p)}$ is similar. The right to left inclusion is trivial. As for inclusion in other direction, suppose there exists $J\in \val{V^{\mathsf{FI}}(p)}^\uparrow$ and $J\not\in \descr{V^{\mathsf{FI}}(p)}$. Thus, $(\val{p},\descr{p}) \not\in J$ and for any filter $F$ which contains $(\val{p},\descr{p})$, $F\cap J\neq \emptyset$. As $(\val{p},\descr{p}) \not\in J$,  by the filter-ideal theorem, there exists a filter $F'$ such that $(\val{p},\descr{p}) \in F'$ and $F'\cap J=\emptyset$. This  contradicts our assumption. So, the left to right inclusion holds. 
\end{remark}
The following lemma shows that the filter-ideal extensions of $\mathrm{LE}$-models preserve modal satisfaction.
\begin{lemma}\label{lem:Filter-ideal satisfaction}
Let $\mathfrak{M}=(\mathfrak{F},V)$ be any $\mathrm{LE}$-model where $\mathfrak{F}=(A,X,I, R_\Box, R_\Diamond)$, and $\mathfrak{M}^{\mathsf{FI}} =(\mathfrak{F}^{\mathsf{FI}},V^{\mathsf{FI}})$ is the filter-ideal extension of $\mathfrak{M}$. Let $F$ be a filter of $\mathfrak{F}^+$ and $J$ be an ideal of $\mathfrak{F}^+$. Then
\smallskip

{{\centering 
$\mathfrak{M}^{\mathsf{FI}}, F\Vdash\varphi$ iff $(\val{\varphi},\descr{\varphi})\in F$
\quad and \quad
$\mathfrak{M}^{\mathsf{FI}}, J\succ\varphi$ iff $(\val{\varphi},\descr{\varphi})\in J$.
\par}}
\end{lemma}

\begin{proof}
It is straight-forward by induction on the complexity of formulas.
\end{proof}

\noindent For any $\mathrm{LE}$-frame $\mathfrak{F}$ and $\mathrm{K}\subseteq \mathfrak{F}^+$, let 
$\mathsf{Fi}(\mathrm{K})$ (resp.~$\mathsf{Id}(\mathrm{K})$) denote the filter (resp.~the ideal) generated by $\mathrm{K}$. When $\mathrm{K}=\{k\}\subseteq \mathfrak{F}^+$, let $\mathsf{Fi}(k)$ (resp.~$\mathsf{Id}(k)$) denote the principal filter (resp.\ principal ideal) $\mathsf{Fi}(\{k\})$ (resp.~$\mathsf{Id}(\{k\})$).

\begin{corollary}[truth lemma]\label{cor:truth lemma}
Let $\mathfrak{M}=(\mathfrak{F},V)$ be any $\mathrm{LE}$-model where $\mathfrak{F}=(A,X,I, R_\Box, R_\Diamond)$, and $\mathfrak{M}^{\mathsf{FI}}$ be the filter-ideal extension of $\mathfrak{M}$. For all $\varphi\in\mathcal{L}$,  

1.~For any $a \in A$, $\mathfrak{M}, a \Vdash \varphi$ if and only if $\mathfrak{M}^{\mathsf{FI}}, \mathsf{Fi}(\mathbf{a}) \Vdash \varphi$.

2.~For any $x \in X$, $\mathfrak{M}, x \succ \varphi$ if and only if $\mathfrak{M}^{\mathsf{FI}}, \mathsf{Id}(\mathbf{x}) \succ \varphi$.

\end{corollary}

\begin{proof}

\smallskip
{{
\centering
\begin{tabular}{ccccccc}
  $ \mathfrak{M}, a\Vdash\varphi $ & $\Leftrightarrow $ & $\mathbf{a}\leq (\val{\varphi}, \descr{\varphi})$ & \quad & $ \mathfrak{M},x\succ\varphi$ &$\Leftrightarrow $   &$(\val{\varphi}, \descr{\varphi})\leq \mathbf{x}$ \\
          & $\Leftrightarrow$ & $ (\val{\varphi}, \descr{\varphi})\in \mathsf{Fi}(\mathbf{a})$ & \quad &  & $\Leftrightarrow$ & $ (\val{\varphi}, \descr{\varphi})\in \mathsf{Id}(\mathbf{x})$ \\
        & $\Leftrightarrow$ & $\mathfrak{M}^{\mathsf{FI}}, \mathsf{Fi}(\mathbf{a}) \Vdash \varphi$ & \quad & & $\Leftrightarrow $ & $\mathfrak{M}^{\mathsf{FI}}, \mathsf{Id}(\mathbf{x}) \succ \varphi$ \\
\end{tabular}
\par
}}
\smallskip

\noindent The first equivalences follow from the relationship between $\Vdash$ and $\succ$ on a polarity-based model and order $\leq$ on its complex algebra. The second equivalence follows from the definitions of $\mathsf{Fi}(\mathbf{a})$ and $\mathsf{Id}(\mathbf{x})$, and the third equivalence follows from Lemma \ref{lem:Filter-ideal satisfaction}.
\end{proof}

The above lemma implies that for any $\mathrm{LE}$-model $\mathfrak{M}$ with filter-ideal extension $\mathfrak{M}^{\mathsf{FI}}$, for any $a\in A$, and $x\in X$, 
 $a \leftrightsquigarrow_A\mathsf{Fi}(\mathbf{a})$ and $x \leftrightsquigarrow_X\mathsf{Id}(\mathbf{x})$. %The following lemma shows that the filter-ideal extension of any $\mathrm{LE}$-model is $\mathrm{M}$-saturated. 

\begin{lemma}\label{lem:extension_M-saturaed}
The filter-ideal extension of an $\mathrm{LE}$-model is $\mathrm{M}$-saturated.
%For any $\mathrm{LE}$-model $\mathfrak{M}$, its filter-ideal extension $\mathfrak{M}^{\mathsf{FI}}$ is $\mathrm{M}$-saturated. 
\end{lemma}
\begin{proof}
 See Appendix \ref{appendix:Proof filter-ideal}.
\end{proof}

This lemma gives us the following  theorem and corollary immediately, saying  filter-ideal extensions $\mathrm{LE}$-models have the Hennessy-Milner property.

\begin{theorem}
   For all $\mathrm{LE}$-models $\mathfrak{M}_1$ and $\mathfrak{M}_2$, $a_1, x_1$ in $\mathfrak{M}_1$, and $a_2, x_2$ in $\mathfrak{M}_2$,
   
1.~$\mathfrak{M}_1, a_1\rightsquigarrow_A \mathfrak{M}_2, a_2$ if and only if $\mathfrak{M}^\mathsf{Fi}_1, \mathsf{Fi}(\mathbf{a_1})\rightrightarrows\mathfrak{M}^\mathsf{Fi}_2, \mathsf{Fi}(\mathbf{a_2})$.

2.~$\mathfrak{M}_1, x_1\rightsquigarrow_X \mathfrak{M}_2, x_2$ if and only if $\mathfrak{M}^\mathsf{Id}_1, \mathsf{Id}(\mathbf{x_1})\leftleftarrows\mathfrak{M}^\mathsf{Id}_2, \mathsf{Id}(\mathbf{x_2})$.
\end{theorem}

\begin{proof}
Immediate by Corollary \ref{cor:truth lemma}, Lemma \ref{lem:extension_M-saturaed}, and Theorem \ref{Thm:M-saturated is HM}.
\end{proof}

\begin{corollary}
For all $\mathrm{LE}$-models $\mathfrak{M}_1$ and $\mathfrak{M}_2$, any $a_1, x_1$ in $\mathfrak{M}_1$ and any $a_2, x_2$ in $\mathfrak{M}_2$,  
  
1.~$\mathfrak{M}_1, a_1\leftrightsquigarrow_A \mathfrak{M}_2, a_2$ if and only if $\mathfrak{M}^\mathsf{Fi}_1, \mathsf{Fi}(\mathbf{a_1})\rightleftarrows\mathfrak{M}^\mathsf{Fi}_2, \mathsf{Fi}(\mathbf{a_2})$.

2.~$\mathfrak{M}_1, x_1\leftrightsquigarrow_X \mathfrak{M}_2, x_2$ if and only if $\mathfrak{M}^\mathsf{Id}_1, \mathsf{Id}(\mathbf{x_1})\rightleftarrows\mathfrak{M}^\mathsf{Id}_2, \mathsf{Id}(\mathbf{x_2})$.
\end{corollary}

\iffalse
\begin{theorem}
    Let $\mathfrak{M}_1$ and $\mathfrak{M}_2$ be any $\mathrm{LE}$-models. Then for $a_1, x_1 \in \mathfrak{M}_1$ and $a_2, x_2 \in \mathfrak{M}_2$ 

    \begin{enumerate}
        \item[1.] $a_1\leftrightsquigarrow a_2$ if and only if $Fil(\mathbf{a_1})$ and $Fil(\mathbf{a_2})$ are bisimilar.
        \item[2.] $x_1\leftrightsquigarrow x_2$ if and only if $Id(\mathbf{x_1})$ and $Id(\mathbf{x_2})$ are bisimilar.
    \end{enumerate}
That is, two points are modally equivalent, if and only if they are bisimilar in their filter-ideal models.
\end{theorem}
\fi

\section{Toward van Benthem Characterization Theorem}\label{sec:Toward van Benthem Characterization Theorem}

In this section, we firstly introduce the standard translation for non-distributive modal logic based on polarity-based semantics, and then define the ultrapower extension for $\mathrm{LE}$-models. Following  that, we introduce $\omega$-saturated $\mathrm{LE}$-models and show that $\omega$-saturated $\mathrm{LE}$-models are $\mathrm{M}$-saturated. Finally, we show that for any $\mathrm{LE}$-model,  there exists an ultrapower extension of it which is $\omega$-saturated.  We use this result to prove the van Benthem characterization theorem of non-distributive modal logic.

%which is that non-distributive modal logic is the bisimilarity invariant fragment of two sorted first-order logic.

\subsection{Standard Translation}\label{subsec:Standard Translation}

The following definition of standard translation for non-distributive modal logic on polarity-based semantics is inspired from \cite[Section 2.1]{CONRADIE2019923}. 
\begin{definition}
Let $g\in G$ and $m\in M$ be two $\mathcal{L}_1$-variables and $\varphi\in\mathcal{L}$ (cf.\ Section \ref{Two Sorted First-Order Logic}). The {\em standard translations} $ST_g:\mathcal{L}\rightarrow\mathcal{L}^1$ and $ST_m: \mathcal{L}\rightarrow\mathcal{L}^1$ are defined as follows:

\smallskip
{{
\centering
\small{
\begin{tabular}{lcl}
  $ST_g(\bot) := \forall m( g I m)$   & \quad &  $ST_m(\bot) := m = m$ \\
   $ST_g(\top) := g=g$   & \quad &  $ST_m(\bot) :=  \forall g (g I m)$ \\
    $ST_g(p) := P_A(g)$   & \quad &  $ST_m(p) := P_X(m)$ \\
     $ST_g(\varphi \vee \psi ) :=\forall m (ST_m(\varphi) \wedge ST_m(\psi) \rightarrow g I m) $   & \quad &  $ST_m(\varphi \vee \psi ) :=ST_m(\varphi) \wedge ST_m(\psi)$ \\
     $ST_g(\varphi \wedge \psi ) :=ST_g(\varphi) \wedge ST_g(\psi)$   & \quad & $ST_m(\varphi \wedge \psi ) :=\forall g (ST_g(\varphi) \wedge ST_g(\psi) \rightarrow g I m) $\\
       $ST_g(\Diamond \varphi) := \forall m (ST_m(\Diamond \varphi)  \rightarrow g I m) $   & \quad &  $ST_m(\Diamond \varphi ) :=\forall g (ST_g(\varphi)  \rightarrow m R_\Diamond g)$ \\
       $ST_g(\Box \varphi ) :=\forall m (ST_m(\varphi)  \rightarrow g R_\Box  m)$  & \quad &  $ST_m(\Box \varphi) := \forall g (ST_g(\Box \varphi)  \rightarrow g I m) $\\
\end{tabular}
\par
}
}}
\end{definition}
The following lemma is immediate from the definition above. 
\begin{lemma}\label{lem:standard translation}
   Let  $\mathfrak{M}=(A, X, I, R_\Box, R_\Diamond, V)$ be an $\mathrm{LE}$-model, then for any $a\in A$ and $x \in X$, and any $\mathcal{L}_1$-variables $g\in G$ and $m\in M$:
  
1.~$\mathfrak{M}, a \Vdash \varphi$ iff $\mathfrak{M} \vDash ST_g(\varphi)[a]$ \quad and \quad $\mathfrak{M}, x \succ  \varphi$ iff $\mathfrak{M} \vDash ST_m(\varphi)[x]$.

2.~$\mathfrak{M}, a\Vdash \varphi$ iff $\mathfrak{M} \vDash \forall g ST_g(\varphi)$ \quad and \quad $\mathfrak{M}, x \succ \varphi$ iff $\mathfrak{M} \vDash \forall m ST_m(\varphi)$.
\end{lemma}

\begin{proof}
    It is straight-forward by induction on the complexity of formulas.
\end{proof}

\subsection{Ultrapower Extension and \texorpdfstring{$\omega$}{omega}-Saturated Model}\label{subsec:Ultrapower Extension and omega Saturated Model}

\begin{definition}
Let $\mathrm{K}$ be a non-empty index set, and $\mathfrak{M}=(\mathfrak{F}, V)$ be an $\mathrm{LE}$-model, where $\mathfrak{F}=(A, X, I, R_\Box, R_\Diamond)$ and for any $p\in\atprop$, $V(p)=(\val{p},\descr{p})$. The {\em $\mathrm{K}$-power} of $\mathfrak{M}$ is the $\mathrm{LE}$-model $\mathfrak{M}^{\mathrm{K}}=(A^{\mathrm{K}}, X^{\mathrm{K}}, I^{\mathrm{K}}, R^{\mathrm{K}}_\Box, R^{\mathrm{K}}_\Diamond, V^{\mathrm{K}})$, where

1.~$A^{\mathrm{K}}$ is the set of functions $s: \mathrm{K}\rightarrow A$.

2.~$X^{\mathrm{K}}$ is the set of functions from $t:\mathrm{K}\rightarrow X$.

3.~For any $s\in A^{\mathrm{K}}$ and $t\in X^{\mathrm{K}}$, $sI^{\mathrm{K}} t$ iff  for any $k\in\mathrm{K}$, $s(k)It(k)$.

 4.~For any $s\in A^{\mathrm{K}}$ and $t\in X^{\mathrm{K}}$, $sR^{\mathrm{K}}_\Box t$ iff for any $k\in\mathrm{K}$, $s(k)R_\Box t(k)$.
 
5.~For any $s\in A^{\mathrm{K}}$ and $t\in X^{\mathrm{K}}$, $tR^{\mathrm{K}}_\Diamond s$ iff for any $k\in\mathrm{K}$, $t(k)R_\Diamond s(k)$.

6.~For any $p\in\atprop, V^{\mathrm{K}}(p)=(\val{p}^{\mathrm{K}},\descr{p}^{\mathrm{K}})$ such that $s\in\val{p}^{\mathrm{K}}$ iff for any $k\in\mathrm{K}$, $s(k)\in\val{p}$, and $t\in\descr{p}^{\mathrm{K}}$ iff for any $k\in\mathrm{K}$, $t(k)\in\descr{p}$.

Let $\mathcal{U}\subseteq\mathcal{P}(\mathrm{K})$ be an ultrafilter over $\mathrm{K}$. We define equivalence relations $\sim_A$ and $\sim_X$ on $A^{\mathrm{K}}$ and $X^{\mathrm{K}}$,  respectively,  as follows:

1.~$s_1\sim_A s_2$ if and only if $\{k\in\mathrm{K}\mid s_1(k)=s_2(k)\}\in\mathcal{U}$.

2.~$t_1\sim_X t_2$ if and only if $\{k\in\mathrm{K}\mid t_1(k)=t_2(k)\}\in\mathcal{U}$.

The {\em ultrapower} of $\mathfrak{M}$ modulo $\mathrm{K}$ and $\mathcal{U}$ is a structure $\mathfrak{M}^{\mathrm{K}}_\mathcal{U}=(A^{\mathrm{K}}_\mathcal{U}, X^{\mathrm{K}}_\mathcal{U}, I^{\mathrm{K}}_\mathcal{U}, {R^{\mathrm{K}}_\Box}_\mathcal{U}, {R^{\mathrm{K}}_\Diamond}_\mathcal{U}, V^{\mathrm{K}}_\mathcal{U})$, where $A^{\mathrm{K}}_\mathcal{U}$ (resp.~$X^{\mathrm{K}}_\mathcal{U}$) is quotient of $A^{\mathrm{K}}$ (resp.~$X^{\mathrm{K}}$) over $\sim_A$ (resp.~$\sim_X$), and for any $[s]\in A^{\mathrm{K}}_\mathcal{U}$, $[t]\in X^{\mathrm{K}}_\mathcal{U}$,

1.~$[s]I^{\mathrm{K}}_\mathcal{U}[t]$ if and only if $\{k\in\mathrm{K}\mid s(k)It(k)\}\in\mathcal{U}$.

2.~$[s]{R^{\mathrm{K}}_\Box}_\mathcal{U}[t]$ if and only if $\{k\in\mathrm{K}\mid s(k)R_\Box t(k)\}\in\mathcal{U}$.

3.~$[t]{R^{\mathrm{K}}_\Diamond}_\mathcal{U}[s]$ if and only if $\{k\in\mathrm{K}\mid t(k){R_\Diamond}s(k)\}\in\mathcal{U}$.

4.~For any $p\in\atprop$, $V^{\mathrm{K}}_\mathcal{U}(p)=(\val{p}^{\mathrm{K}}_\mathcal{U}, \descr{p}^{\mathrm{K}}_\mathcal{U})$, where $[s]\in\val{p}^{\mathrm{K}}_\mathcal{U}$ if and only if $\{k\in\mathrm{K}\mid s(k)\in\val{p}\}\in\mathcal{U}$, and $[t]\in\descr{p}^{\mathrm{K}}_\mathcal{U}$ if and only if $\{k\in\mathrm{K}\mid t(k)\in\descr{p}\in\mathcal{U}\}$.
\end{definition}

The following theorem is an immediate consequence of \L\'{o}s's Theorem for first-order logic. 
\begin{theorem}[\L\'{o}s's Theorem]
Let $\mathrm{K}$ be a non-empty index set and $\mathcal{U}$ be an ultrafilter over $\mathrm{K}$. Let $\mathfrak{M}$ be an $\mathrm{LE}$-model and $\mathfrak{M}^\mathrm{K}_\mathcal{U}$ be its ultrapower modulo $\mathrm{K}$ and $\mathcal{U}$. For any two-sorted first-order formula $\varphi\in \mathcal{L}^1$ (cf.\ Section \ref{Two Sorted First-Order Logic}),

{{
\centering
 $\mathfrak{M}^\mathrm{K}_\mathcal{U}\vDash\varphi[s,t]$ \quad if and only if \quad  $\{k\in\mathrm{K}\mid \mathfrak{M}\vDash\varphi[s(k),t(k)]\}\in\mathcal{U}$.
 \par
}}

\end{theorem}

\begin{remark}
Any ultrapower $\mathfrak{M}^{\mathrm{K}}_\mathcal{U}$ of an $\mathrm{LE}$-model $\mathfrak{M}$ is indeed an $\mathrm{LE}$-model. In order to see this, we just need to check that relations ${R^{\mathrm{K}}_\Box}_\mathcal{U}$ and ${R^{\mathrm{K}}_\Diamond}_\mathcal{U}$ are $I$-compatible, and for any $p\in\atprop$, $\val{p}^{\mathrm{K}}_\mathcal{U}$ and $\descr{p}^{\mathrm{K}}_\mathcal{U}$ are Galois stable. This is true because all these conditions are first-order definable, and so they will be preserved in the ultrapower extension by \L\'{o}s's Theorem.
\end{remark}
We now define $\kappa$-saturated first-order models for any infinite cardinal $\kappa$.
\begin{definition}
    Let $\kappa$ be an infinite cardinal, and  $\mathfrak{M}$ be an $\mathrm{LE}$-model.  For any set  $S$ of elements  appearing in $\mathfrak{M}$, let $\mathcal{L}^1_{S}$ be the extension of $\mathcal{L}^1$ with constant symbols for every element of $S$.   $\mathfrak{M}$ is {\em $\kappa$-saturated} if for any set of elements $S$ appearing in $\mathfrak{M}$  with $|S| < \kappa$,  and  
    any set  $\Sigma$ of $\mathcal{L}^1_{S}$ formulas  with finite many free variables, if  $\Sigma$ is finitely satisfiable in $\mathfrak{M}$, then $\Sigma$ is satisfiable in $\mathfrak{M}$. 
\end{definition}

The following theorem follows from   \cite[Theorem 6.1.8]{chang1990model}. 

\begin{theorem}\label{thm:ultrapower extension is omega-saturated}
For any $\mathrm{LE}$-model $\mathfrak{M}$, there exist an non-empty index set $\mathrm{K}$ and an ultrafilter $\mathcal{U}$ over $\mathrm{K}$ such that $\mathfrak{M}^\mathrm{K}_\mathcal{U}$ is $\omega$-saturated.
\end{theorem}

\begin{proposition}\label{pro: Any omega-saturated LE-model is M-saturated}
    Any $\omega$-saturated $\mathrm{LE}$-model is $\mathrm{M}$-saturated. 
\end{proposition}
\begin{proof}
    Let $\mathfrak{M}= (\mathfrak{F},V)$ be an $\omega$-saturated $\mathrm{LE}$-model, where $\mathfrak{F}=(A,X,I,R_\Box, R_\Diamond)$. We only prove the fist item of Definition \ref{def:M-saturation}, and other conditions can be proved similarly. Let $\Sigma$ be a set of $\mathcal{L}$-formulas which are finitely satisfiable in set $\{x' \mid a I^c x'\}$. Define $\Sigma'$ to be the set 

{{
\centering
$ \Sigma'= \{a I^c x \} \cup ST_x(\Sigma)$.
\par
}}

    Clearly $\Sigma'$ contains only finite many free variables (only $x$)  and  every finite subset of  $\Sigma'$ is  satisfiable in $(\mathfrak{M},x)_{x \in X} $. Hence, by $\omega$-saturation $\Sigma'$ is  satisfiable in $(\mathfrak{M},x)_{x \in X} $ at some $x_0$ with $a I^c x_0$ . Then, by  Lemma \ref{lem:standard translation} we have $\mathfrak{M},x_0 \succ \Sigma $. Therefore, $\Sigma$ is satisfiable in  $\{x' \mid a I^c x'\}$.
\end{proof}
It follows immediately from the above proposition and Theorem \ref{Thm:M-saturated is HM} that the class of $\omega$-saturated LE-models satisfy Hennessy-Milner property. 

\subsection{The van Benthem Characterization Theorem}\label{subsec:The van Benthem Characterization Theorem}

We are now ready to prove the van Benthem characterization theorem for non-distributive modal logic based on the polarity-based semantics, which follows from the detour lemma below.

\begin{theorem}[Detour Lemma]\label{thm:Detour Lemma}
For all $\mathrm{LE}$-models $\mathfrak{M}_1$ and $\mathfrak{M}_2$, any $a_1, x_1$ in $\mathfrak{M}_1$ and $a_2, x_2$ in $\mathfrak{M}_2$,
%Let $\mathfrak{M}_1$ and $\mathfrak{M}_2$ be any $\mathrm{LE}$-models such that $a_1, x_1$ are in $\mathfrak{M}_1$ and $a_2, x_2$  are in $\mathfrak{M}_2$. Then,  the following propositions hold.

1.~$\mathfrak{M}_1, a_1\rightsquigarrow_A \mathfrak{M}_2, a_2$ if and only if there exist two $\omega$-saturated $\mathrm{LE}$-models $\mathfrak{M}^\omega_1$ and $\mathfrak{M}^\omega_2$, and two elementary embeddings $\omega_1:\mathfrak{M}_1\rightarrow\mathfrak{M}^\omega_1$ and $\omega_2:\mathfrak{M}_2\rightarrow\mathfrak{M}^\omega_2$, such that $\mathfrak{M}^\omega_1, \omega_1(a_1)\rightrightarrows\mathfrak{M}^\omega_2, \omega_2(a_2)$.
     %if and only if $\mathfrak{M}^\mathsf{FI}_1, \mathsf{Fi}(\mathbf{a}_2)$ and $\mathfrak{M}^\mathsf{FI}_2, \mathsf{Fi}(\mathbf{a}_2)$ are bisimilar 
     
 2.~$\mathfrak{M}_1, x_1\rightsquigarrow_X \mathfrak{M}_2, x_2$ if and only if there exist two $\omega$-saturated $\mathrm{LE}$-models $\mathfrak{M}^\omega_1$ and $\mathfrak{M}^\omega_2$, and two elementary embeddings $\omega_1:\mathfrak{M}_1\rightarrow\mathfrak{M}^\omega_1$ and $\omega_2:\mathfrak{M}_2\rightarrow\mathfrak{M}^\omega_2$, such that $\mathfrak{M}^\omega_1, \omega_1(x_1)\leftleftarrows\mathfrak{M}^\omega_2, \omega_2(x_2)$.
\end{theorem}

\begin{figure}
    \centering
        \includegraphics[scale=0.1]{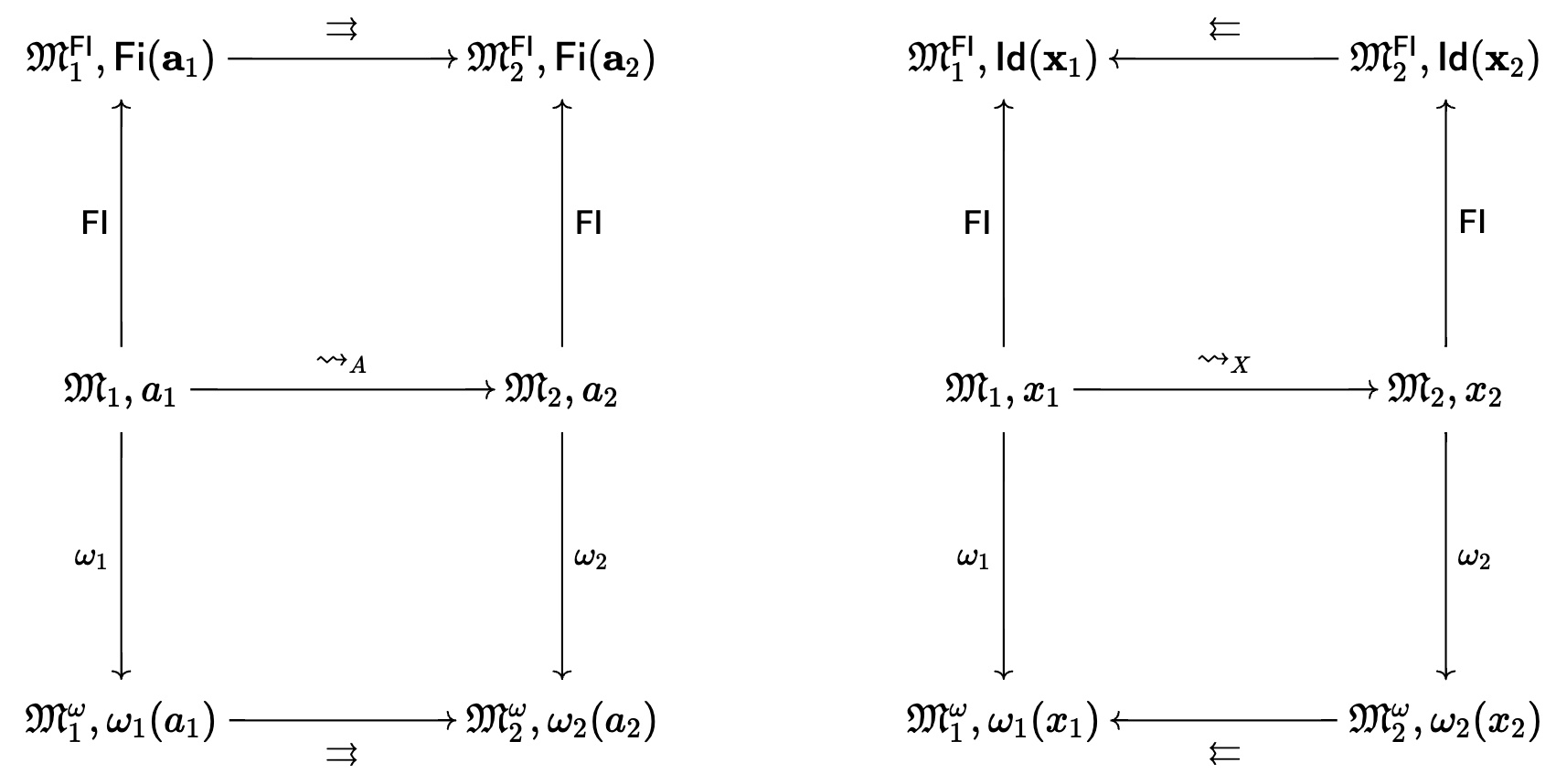}
        \vspace{-0.2cm}
    \caption{Detour Lemma}
    \vspace{-0.3cm}
    \label{fig:counter-example 3}
\end{figure}
\begin{proof}
    We only prove the item 1  here,  the proof of the item 2  is similar. The implication from  right to left proved by contradiction. Suppose   $\mathfrak{M}^\omega_1, \omega_1(a_1) \rightrightarrows\mathfrak{M}^\omega_2, \omega_2(a_2)$,  and $\omega_1:\mathfrak{M}_1\rightarrow\mathfrak{M}^\omega_1$,  $\omega_2:\mathfrak{M}_2\rightarrow\mathfrak{M}^\omega_2$ are two elementary embeddings. Suppose there exists $\varphi\in\mathcal{L}$, such that $\mathfrak{M}, a_1\Vdash\varphi$, and $\mathfrak{M}, a_2\nVdash\varphi$. Therefore, by Lemma \ref{lem:standard translation}, $\mathfrak{M}_1\vDash ST_g(\varphi)[a_1]$ and $\mathfrak{M}_2\nVdash ST_g(\varphi)[a_2]$, which implies that $\mathfrak{M}^\omega_1\vDash ST_g(\varphi)[\omega_1(a_1)]$ and $\mathfrak{M}^\omega_2\nvDash ST_g(\varphi)[\omega_2(a_2)]$. Therefore, by Lemma \ref{lem:standard translation} again, $\mathfrak{M}^\omega_1, \omega_1(a_1)\Vdash\varphi$ and $\mathfrak{M}^\omega_2, \omega_2(a_2)\nVdash\varphi$. However, this contradicts   $\mathfrak{M}^\omega_1, \omega_1(a_1) \rightrightarrows\mathfrak{M}^\omega_2, \omega_2(a_2)$ due to Theorem  \ref{thm:finite image implies bisimilar}. Hence proved. 
    
    For the implication from left to right, assume that $\mathfrak{M}_1, a_1\rightsquigarrow_A \mathfrak{M}_2, a_2$. We consider their respective ultrapower extensions ${\mathfrak{M}^\mathrm{K}_1}_\mathcal{U}$, ${\mathfrak{M}^\mathrm{K}_2}_\mathcal{U}$, and natural elementary embeddings $\omega_i:\mathfrak{M}_i\rightarrow{\mathfrak{M}^\mathrm{K}_i}_\mathcal{U}$ where $i\in\{1,2\}$
    such that for any $a_i\in A_i$, $\omega_i(a_i)=[s^i_{a_i}]$, where $s^i_{a_i}:\mathrm{K}\rightarrow A_i$ are the constant functions which send all elements in $\mathrm{K}$ to $a_i$. Therefore, ${\mathfrak{M}^\mathrm{K}_1}_\mathcal{U}, \omega_1(a_1)\rightsquigarrow_A {\mathfrak{M}^\mathrm{K}_2}_\mathcal{U}, \omega_2(a_2)$. By Theorem \ref{thm:ultrapower extension is omega-saturated}, ${\mathfrak{M}^\mathrm{K}_1}_\mathcal{U}$ and ${\mathfrak{M}^\mathrm{K}_2}_\mathcal{U}$ are  $\omega$-saturated $\mathrm{LE}$-models. Therefore, by Proposition \ref{pro: Any omega-saturated LE-model is M-saturated} and Theorem \ref{Thm:M-saturated is HM}, ${\mathfrak{M}^\mathrm{K}_1}_\mathcal{U}, \omega_1(a_1)\rightrightarrows{\mathfrak{M}^\mathrm{K}_2}_\mathcal{U}, \omega_2(a_2)$.
\end{proof}

\begin{theorem}[van Benthem Characterization Theorem]
Let $\varphi(g)$ (resp.~$\varphi(m)$) be any two sorted first-order formula in $\mathcal{L}^1$, then $\varphi(g)$ (resp.~$\varphi(m)$) is preserved (resp.~reflected) by  simulation if and only if there exist finite many $\mathcal{L}$-formulas $\psi_1$, $\psi_2$, $\cdots$, $\psi_n$ such that $\varphi(g)$ (resp.~$\varphi(m)$) is equivalent to $\bigvee_{1\leq i\leq n} ST_g(\psi_i)$ (resp.~$\bigvee_{1\leq i\leq n}  ST_m(\psi_i)$).
\end{theorem}

\begin{proof}
We only provide the proof for $\varphi(g)$, the  the proof for $\varphi(m)$ is similar. The implication from right to left follows from Theorem  \ref{thm:simulation invariance}. For the implication from left to right, assume that $\varphi(g)$ is preserved by simulation. We define a set of two sorted first-order formulas $\mathrm{JMOC}(\varphi):=\{\bigvee_{1\leq i\leq n}ST_g(\psi_i)\mid\psi_i\in\mathcal{L}, n\in\omega, \varphi(g)\vDash\bigvee_{1\leq i\leq n}ST_g(\psi_i)\}$. We first check that the following is true: if $\mathrm{JMOC}(\varphi)\vDash\varphi(g)$, then $\varphi(g)$ is equivalent to $\bigvee_j ST_g(\psi_j)$ for some finite
set of $\mathcal{L}$-formulas $\{\psi_j \mid j \in \mathcal{J}\}$. Assume that $\mathrm{JMOC}(\varphi)\vDash\varphi(g)$. By compactness, there exists an finite set $\{\bigvee_{1\leq i\leq n}ST_g(\psi_{j,i})\mid j\in \mathcal{J}, n\in\omega\}\subseteq\mathrm{JMOC}(\varphi)$ such that $\vDash\bigwedge_{j\in\mathcal{J}}\bigvee_{1\leq i\leq n}ST_g(\psi_{j,i})\rightarrow\varphi(g)$. Therefore, we have $\vDash\varphi(g)\leftrightarrow\bigwedge_{j\in\mathcal{J}}\bigvee_{1\leq i\leq n}ST_g(\psi_{j,i})$.  Note that $\bigwedge_{j\in\mathcal{J}}\bigvee_{1\leq i\leq n}ST_g(\psi_{j,i})$ is equivalent to $\bigvee_{1\leq i\leq n}\bigwedge_{j\in\mathcal{J}}ST_g(\psi_{j,i})$, and the latter is equivalent to $\bigvee_{1\leq i\leq n}ST_g(\bigwedge_{j\in\mathcal{J}}\psi_{j,i})$.

Therefore, it suffices to show that $\mathrm{JMOC}(\varphi)\vDash\varphi(g)$. Let $\mathfrak{M}$ be an $\mathrm{LE}$-model such that $\mathfrak{M}\vDash\mathrm{JMOC}(\varphi)[a]$. Let $\mathrm{NT}(g):= \{\neg ST_g(\psi) \mid \psi\in\mathcal{L}, \mathfrak{M}\nvDash ST_g(\psi)[a]\}$. We want to show that $\mathrm{NT}(g) \cup \{\varphi(g)\}$ is consistent. Suppose not. Then, by compactness there exists a finite subset $\mathrm{NT}_0(g)=\{\neg ST_g(\psi_k)\mid k\in\mathcal{K}\}$ of $\mathrm{NT}(g)$ such that $\mathrm{NT}_0(g) \cup \{\varphi(g)\}$ is inconsistent. Therefore, $\varphi(g)\rightarrow \neg \bigwedge_{k\in\mathcal{K}}\neg ST_g(\psi_i) $ is valid, which means that $\varphi(g)\rightarrow \bigvee_{k\in\mathcal{K}}  ST_g(\psi_i)$ is valid. Therefore, $\bigvee_{k\in\mathcal{K}}  ST_g(\psi_i)\in\mathrm{JMOC}(\varphi)$, which implies $\mathfrak{M}\vDash \bigvee_{k\in\mathcal{K}}  ST_g(\psi_i)[a]$. However, $\mathfrak{M}\nvDash ST_g(\psi_i)[a]$ for any $k\in\mathcal{K}$, which implies $\mathfrak{M}\nvDash \bigvee_{k\in\mathcal{K}}  ST_g(\psi_i)[a]$, so there is a contradiction. Hence, $\mathrm{NT}(g) \cup \{\varphi(g)\}$ must be consistent. Therefore, there exists an $\mathcal{L}^1$-structure $\mathfrak{N}$ and $a'$ in $\mathfrak{N}$ such that $\mathfrak{N}\vDash\mathrm{NT}(g)\cup\{\varphi(g)\}[a']$. 

For any formula $\psi$, $\mathfrak{M}, a\nVdash\psi$ iff $\neg ST_g(\psi)\in NT(g)$, which implies $\mathfrak{N}, a'\nVdash\psi$. Therefore, $\mathfrak{N}, a'\rightsquigarrow_A \mathfrak{M}, a$. By Theorem \ref{thm:Detour Lemma} (Detour Lemma), there exist  two $\omega$-saturated $\mathrm{LE}$-models $\mathfrak{N}^\omega$ and $\mathfrak{M}^\omega$, and two elementary embeddings $\omega_n:\mathfrak{N}\rightarrow\mathfrak{N}^\omega$ and $\omega_m:\mathfrak{M}\rightarrow\mathfrak{M}^\omega$, such that $\mathfrak{N}^\omega, \omega_n(a')\rightrightarrows\mathfrak{M}^\omega, \omega_m(a)$. As $\omega_n$ is an elementary embedding and $\mathfrak{N}\vDash\varphi(g)[a']$, we have $\mathfrak{N}^\omega\vDash\varphi(g)[\omega_n(a')]$. Therefore, as $\varphi(g)$ is preserved by simulation, $\mathfrak{M}^\omega\vDash\varphi(g)[\omega_m(a)]$. As $\omega_m$ is an elementary embedding, we have $\mathfrak{M}\vDash\varphi(g)[a]$. This concludes the proof.
\end{proof}

\section{Conclusion}\label{sec:Conclusion}

In this paper, we introduced bisimulations on polarity-based semantics for non-distributive modal logic, demonstrating their application in proving the Hennessy-Milner and van Benthem theorems. These results show that the non-distributive modal logic is bisimulation invariant fragment (within a specific signature and axioms) of two-sorted first-order logic. This work suggests the following directions for future research.

\textbf{Characterizing theorem for lattice-based modal $\mu$-calculus:} Lattice based modal  $\mu$-calculus, i.e.~the extension of non-distributive modal logic with least and greatest fixed point operators, was defined and studied in \cite{ding2023game}. In \cite{ding2023game}, it was showed that the formulas of  modal $\mu$-calculus  are bisimulation-invariant. The classical modal $\mu$-calculus is characterized as the  bisimulation invariant fragment of monadic second-order logic. It would be interesting to see if a similar characterization holds in its lattice-based generalization.

\textbf{Computational properties of bisimulations:}  In classical modal logic, computational complexity of different bisimulation related problems like checking if two models are bisimlar, computing the largest bisimulation on a model have been studied extensively \cite{paige1987three}. It would be interesting to study similar complexity problems in a non-distributive setting and compare the results with the classical setting. 

\textbf{Bisimulations and network analysis:}  In Social Network Analysis (SNA), different notions of equivalences between nodes are defined  to  study  the structural   nodes in one or two social networks  \cite{lorrain1971structural,FAN201466}. On of the most prominent structural equivalences used in SNA is the notion of {\em regular equivalence}, which is analogous to  bisimulations on (multi-modal) Kripke frames \cite{WHITE1983193,MARX200351}. Enriched formal contexts (polarities) can be seen as bi-partite networks; so, it would be interesting to study if simulations and bisimulations defined in this paper can be useful in bi-partite network analysis.

%% Appendix.
%% Remove the \Appendix command if an
%% appendix is not required.
\Appendix

\section{Proofs}

In this appendix, we show the proofs of the results stated throughout the paper.

\subsection{Proof of Theorem \ref{thm:simulation invariance}}\label{Proof of thm:simulation invariance}

\begin{proof}
    We give a proof  by induction on the complexity of the formula $\varphi$. Without loss of generality, let $(S, T)$ be a simulation from $\mathfrak{M}_1$ to $\mathfrak{M}_2$. The initial  step is when $\varphi\in\atprop$ is a propositional variable. For the first item, assume $a_1Sa_2$, then $a_1\rightsquigarrow_A a_2$ follows directly from the item 1 of Definition \ref{def:simulation and bisimulation}. For the second item, assume $x_1Ta_2$, then $x_2\rightsquigarrow_X x_1$ follows directly from the item 2 of Definition \ref{def:simulation and bisimulation}. Now, we consider induction step for all the connectives.

   (1) Suppose $\varphi$ is $\psi_1\wedge\psi_2$.
   
   To prove item 1, suppose  $a_1Sa_2$, and  $\mathfrak{M}_1, a_1\Vdash\psi_1\wedge\psi_2$. Then,   $\mathfrak{M}_1, a_1\Vdash\psi_1$ and $\mathfrak{M}_1, a_1\Vdash\psi_2$, which by induction  hypothesis implies $\mathfrak{M}_2, a_2\Vdash\psi_1$ and $\mathfrak{M}_2, a_2\Vdash\psi_2$, which implies $\mathfrak{M}, a_2\Vdash\psi_1\wedge\psi_2$. 
    
    To prove item 2, suppose  $x_1Tx_2$ and $\mathfrak{M}_1, x_1\nsucc\psi_1\wedge\psi_2$.  Then, there exists $a_1\in A_1$ such that $a_1I^c_1x_1$ and $\mathfrak{M}_1, a_1\Vdash\psi_1\wedge\psi_2$. Then, by item 4 of  Definition \ref{def:simulation and bisimulation}, there exists $a_2\in A_2$ such that $a_2I^c_2x_2$ and $a_1Sa_2$. Therefore, $\mathfrak{M}_2, a_2\Vdash\psi_1\wedge\psi_2$, which implies  $\mathfrak{M}_2, x_2\nsucc\psi_1\wedge\psi_2$.

(2)  Suppose  $\varphi$ is $\psi_1\vee\psi_2$. The proof is dual to the previous case. 
\iffalse 
    Suppose  $\varphi$ is $\psi_1\vee\psi_2$, assume $x_1Tx_2$ for proving the second item. Therefore, there is $\mathfrak{M}_2, x_2\succ\psi_1\vee\psi_2$, which implies $\mathfrak{M}_2, x_2\succ\psi_1$ and $\mathfrak{M}_2, x_2\succ\psi_2$, which implies $\mathfrak{M}_1, x_1\succ\psi_1$ and $\mathfrak{M}_1, a_1\succ\psi_2$ (according to the inductive hypothesis), which implies $\mathfrak{M}, x_1\succ\psi_1\vee\psi_2$. Therefore, there is $x_2\rightsquigarrow_A x_1$. Assume $a_1Sa_2$ for proving the first item. $\mathfrak{M}_2, a_2\nVdash\psi_1\vee\psi_2$ implies there exists $x_2\in X_2$ such that $a_2I^c_2x_2$ and $\mathfrak{M}_2, x_2\succ\psi_1\vee\psi_2$. According to the third condition of Definition \ref{def:simulation and bisimulation}, there exists $x_1\in X_1$ such that $a_1I^c_1 x_1$ and $x_1Tx_2$. Therefore, there is $\mathfrak{M}_1, x_1\succ\psi_1\vee\psi_2$, which gives us $\mathfrak{M}_1, a_1\nVdash\psi_1\vee\psi_2$. Therefore, there is $a_1\rightsquigarrow_A a_2$.
\fi

    (3)  Suppose  $\varphi=\Box\psi$.
    
    To prove item 1, assume $a_1Sa_2$, and $\mathfrak{M}_2, a_2\nVdash\Box\psi$. Then,  there exists $x_2\in X_2$ such that $\mathfrak{M}_2, x_2\succ\psi$ and $a_2R^c_{\Box_2}x_2$.  By item 5 Definition \ref{def:simulation and bisimulation}, there exists $x_1\in X_1$ such that $a_1R^c_{\Box_1}x_1$, and $x_1Tx_2$. Therefore,  by induction hypothesis $\mathfrak{M}_1, x_1\succ\psi$. Hence,  $\mathfrak{M}_1, a_1\nsucc\Box\psi$. 
    
    To prove item 2,  suppose  $x_1Tx_2$, and  $\mathfrak{M}_1, x_1\nsucc\Box\psi$. Then,  there exists $a_1\in A_1$ such that $a_1I^c_1 x_1$ and $\mathfrak{M}_1, a_1\Vdash\Box\psi$. By item 4 of  Definition \ref{def:simulation and bisimulation}, there exists $a_2\in A_2$ such that $a_2I^c_2 x_2$, and $a_1Sa_2$. Therefore, $\mathfrak{M}_2, a_2\Vdash\Box\psi$, which implies $\mathfrak{M}_2, x_2\nsucc\Box\psi$.

    (4)  Suppose  $\varphi$ is $\Diamond \psi$. In this case, the proof is dual to the previous case. 
\iffalse 
    hen $\varphi$ is $\Diamond\psi$, assume $x_1Tx_2$ for the second item. $\mathfrak{M}_1, x_1\nsucc\Diamond\psi$ implies there exists $a_1\in A_1$ such that $x_1R^c_{\Diamond_1}a_1$ and $\mathfrak{M}_1, a_1\Vdash\psi$. According to the sixth condition of Definition \ref{def:simulation and bisimulation}, there exist $a_2\in A_2$ such that $x_2R^c_{\Diamond_2}a_2$ and $a_1Sa_2$. Therefore, $\mathfrak{M}_2, a_2\Vdash\psi$, which gives us $\mathfrak{M}_2, x_2\nsucc\Diamond\psi$. Therefore, there is $x_2\rightsquigarrow_X x_1$. For proving the first item, assume $a_1Sa_2$. $\mathfrak{M}_2, a_2\nVdash\Diamond\psi$ implies there eixsts $x_2\in X_2$ such that $\mathfrak{M}_2, x_2\succ\Diamond\psi$ and $a_2I^cx_2$. According to the third condition, there exists $x_1\in X_1$ such that $a_1I^cx_1$ and $x_1Tx_2$. Therefore, $\mathfrak{M}_1, x_1\succ\Diamond\psi$, which gives us $\mathfrak{M}_1, a_1\nVdash\Diamond\psi$. Therefore, $a_1\rightsquigarrow_A a_2$. 
\fi 
    This concludes the proof.

    %$\mathfrak{M}_2, a_2\nVdash\psi_1\vee\psi_2$ implies there exists $x_2\in X_2$ such that $a_2I^c_2x_2$ and $\mathfrak{M}_2, x_2\succ\psi_1\vee\psi_2$. Therefore, according to the third condition of Definition \ref{def}, 
    
    %Without loss of generality, we assume that 
\end{proof}

\subsection{Proof of Theorem \ref{Thm:M-saturated is HM}} \label{proof of M-saturated is HM}

\begin{proof}
    Let $\mathfrak{M}_1$ and $\mathfrak{M}_2$ be two $\mathrm{M}$-saturated models, and let $a_1, x_1 \in \mathfrak{M}_1$ and $a_2, x_2 \in \mathfrak{M}_2$. We only prove the first item of Definition \ref{def: HMP}, and the proof of the second item is similar. Assume that $a_1\rightsquigarrow_A a_2$, we claim that $(\rightsquigarrow_A, \leftsquigarrow_X)$ is a simulation from $\mathfrak{M}_1$ to $\mathfrak{M}_2$, where $(a, a')\in\rightsquigarrow_A$ (denoted as $a\rightsquigarrow_A a'$) iff $\mathfrak{M}_1, a\rightsquigarrow_A \mathfrak{M}_2, a'$ and $(x, x')\in\leftsquigarrow_X$ (denoted as $x\leftsquigarrow_X x'$) iff $\mathfrak{M}_1, x\leftsquigarrow_X \mathfrak{M}_2, x'$. Hence, it is sufficient to show that  $(\rightsquigarrow_A, \leftsquigarrow_X)$ satisfies item  1-6 of Definition \ref{def:simulation and bisimulation}. 
    
    Items  1 and 2 are easy to check. To prove item 3, assume that $a_1'\rightsquigarrow_A a_2'$ and $a_2'I^c_2 x_2'$. Let $\Sigma$ be the set of all the formulas satisfied  at $x_2'$. Therefore, for any finite subset  $\Delta$  of $\Sigma$,
    $\mathfrak{M}_2, x_2'\succ\bigvee\Delta$. Therefore, $\mathfrak{M}_2, a_2'\nVdash\bigvee\Delta$, which by $a_1'\rightsquigarrow_A a_2'$ implies $\mathfrak{M}_1, a_1'\nVdash\bigvee\Delta$. So, there exists $x_1'$, such that $\mathfrak{M}_1, x_1'\succ\bigvee\Delta$ and $a_1' I^c_1 x_1'$. Hence, $\Sigma$ is finitely satisfiable in $\{x_1'\mid a_1'I^c_1 x_1'\}$.  Therefore, by $\mathrm{M}$-saturation of $\mathfrak{M}_1$, $\Sigma$ is satisfiable in $\{x_1'\mid a_1'I^c_1 x_1'\}$. Therefore, there exists $x_\Sigma$ such that $a_1'I^c_1 x_\Sigma$ and all formulas in $\Sigma$ are satisfied at $x_\Sigma$ which means $x_\Sigma\leftsquigarrow_X x_2'$. The item  4 is proved similarly to item 3.

    %For Condition 4, assume that $x_1'\rightsquigarrow_A x_2'$ and $a_1'I^c_1 x_1'$. Let $\Sigma$ be the set of all the formulas satisfied at $a_1'$. Therefore, for any finite subset $\Delta$  of $\Sigma$,  $\mathfrak{M}_1, a_1'\Vdash\bigwedge\Delta$. Therefore $\mathfrak{M}_1, x_1'\nsucc\bigwedge\Delta$, which by $a_1'\rightsquigarrow_A a_2'$ implies $\mathfrak{M}_2, x_2'\nsucc\bigwedge\Delta$. So there exists $a_2'$ such that $\mathfrak{M}_2, a_2'\Vdash\bigwedge\Delta$ and $a_2' I^c_2 x_2'$. Hence, $\Sigma$ is finitely satisfiable in $\{a_2'\mid a_2'I^c_2 x_2'\}$. Therefore, by $\mathrm{M}$-saturation of $\mathfrak{M}_2$,  $\Sigma$ is satisfiable in $\{a_2'\mid a_2'I^c_2 x_2'\}$. Therefore, there exists $a_\Sigma$ such that $a_\Sigma I^c_2 x_2$ and all formulas in $\Sigma$ are satisfied at $a_\Sigma$, which means  $a_1'\rightsquigarrow_A a_\Sigma$.
    
    %\mpmnote{We can just remove Condition 4 and 6 and say that they are proved similarly to 3 and 5 if needed}
    
    For item 5, assume that $a_1'\rightsquigarrow_A a_2'$ and $a_2'R^c_{\Box_2} x_2'$. Let $\Sigma$ be the set of all the formulas satisfied at $x_2'$. Therefore, for any finite subset $\Delta$ of $\Sigma$, $\mathfrak{M}_2, x_2'\succ\bigvee\Delta$. Therefore $\mathfrak{M}_2, a_2'\nVdash\Box(\bigvee\Delta)$, which by $a_1'\rightsquigarrow_A a_2'$ implies $\mathfrak{M}_1, a_1'\nVdash\Box(\bigvee\Delta)$. So, there exists $x_1'$, such that $\mathfrak{M}_1, x_1'\succ\bigvee\Delta$ and $a_1' R^c_{\Box_1} x_1'$. Hence, $\Sigma$ is finitely satisfiable in $\{x_1'\mid a_1'R^c_{\Box_1} x_1'\}$. Therefore, by $\mathrm{M}$-saturation of $\mathfrak{M}_1$, $\Sigma$ is satisfiable in $\{x_1'\mid a_1'R^c_{\Box_1} x_1'\}$. Therefore, there exists $x_\Sigma$ such that $a_1'R^c_{\Box_1} x_\Sigma$ and all formulas in $\Sigma$ are satisfied at $x_\Sigma$, which means $x_\Sigma\leftsquigarrow_X x_2'$. The item 6 is proved similarly to item 5. This concludes the proof.
    %For Condition 6, assume that $x_1'\rightsquigarrow_A x_2'$ and $x_1'R^c_{\Diamond_1} a_1'$. Let $\Sigma$ be the set of all the formulas satisfied at $a_1'$.  Therefore, for any finite subset $\Delta$ of $\Sigma$, $\mathfrak{M}_1, a_1'\Vdash\bigwedge\Delta$. Therefore $\mathfrak{M}_1, x_1'\nsucc\Diamond(\bigwedge\Delta)$, which by $a_1'\rightsquigarrow_A a_2'$ implies $\mathfrak{M}_2, x_2'\nsucc\Diamond(\bigwedge\Delta)$. So there exists $a_2'$ such that $\mathfrak{M}_2, a_2'\Vdash\bigwedge\Delta$ and $x_2' R^c_{\Diamond_2} a_2'$. Hence, $\Sigma$ is finitely satisfiable in $\{a_2'\mid x_2' R^c_{\Diamond_2} a_2'\}$. Therefore, by $\mathrm{M}$-saturation of $M_2$, $\Sigma$ is satisfiable in $\{a_2'\mid x_2' R^c_{\Diamond_2} a_2'\}$. Therefore, there exists $a_\Sigma$ such that $x_2' R^c_{\Diamond_2} a_\Sigma$ and all formulas in $\Sigma$ are satisfied at $a_\Sigma$, which means $a_1'\rightsquigarrow_A a_\Sigma$.
\end{proof}

\subsection{Proof of Lemma \ref{lem:extension_M-saturaed}} \label{appendix:Proof filter-ideal}
\begin{proof}
Let $\mathfrak{M}$ be an $\mathrm{LE}$-model and $\mathfrak{M}^{\mathsf{FI}}$ be its filter-ideal extension. We only prove items 1 and 3 of Definition \ref{def:M-saturation} hold for $\mathfrak{M}^{\mathsf{FI}}$. The proofs for items 2 and 4 can be obtained dually. 

1.~Let $F$ be any filter of $\mathfrak{F}^+$. Let $\Sigma$ be a set of $\mathcal{L}$-formulas which is finitely satisfiable in $\{J \mid F I^c J\}$.
By definition of filter-ideal extension this set is same as  $\{J \mid F \cap J= \emptyset\}$. By Lemma \ref{lem:Filter-ideal satisfaction},  for any formula $\varphi$,  $\mathfrak{M}^{\mathsf{FI}}, J \succ \varphi$ iff $(\val{\varphi},\descr{\varphi}) \in J$. Thus, $\Sigma$ is finitely satisfiable in
$\{J \mid F I^c J\}$ iff for any finite $\Delta \subseteq \Sigma$,  the ideal  generated by the set $\{(\val{\varphi}, \descr{\varphi}) \mid \varphi \in \Delta\}$  (denoted by $Id(\Delta)$) is disjoint with $F$. 
 Let 
\[Y=\bigcup \{ Id(\Delta) \mid 
\Delta \subseteq \Sigma, \, \Delta \,  \mbox{is finite} \}. 
\]
It is clear that $F \cap Y= \emptyset$ and $(\val{\varphi}, \descr{\varphi}) \in Y$ for all $\varphi \in \Sigma$. We claim that $Y$ is an ideal of $\mathfrak{F}^+$. Suppose $c \in Y$, and $d \leq c$. By definition of $Y$,  $c$ must be in $Id(\Delta)$ for some finite $\Delta \subseteq \Sigma$. Then, we must have $d \in Id(\Delta) \subseteq Y$. Finally, suppose $c_1,c_2 \in Y$. Then, we must have finite $\Delta_1,\Delta_2 \subseteq \Sigma$ such that,  $c_1 \in Id(\Delta_1)$ and $c_2 \in Id(\Delta_2)$. Therefore, $c_1,c_2 \in Id(\Delta_1 \cup \Delta_2)$ implying $c_1 \vee c_2 \in Id(\Delta_1 \cup \Delta_2)$. As $\Delta_1 \cup \Delta_2 \subseteq \Sigma$ is finite, we must have $Id(\Delta_1 \cup \Delta_2) \subseteq Y$. Thus, we have $c_1 \vee c_2 \in Y$. So, $Y$ is an ideal such that $F I^c Y$ and $\mathfrak{M}^{\mathsf{FI}}, Y\succ\Sigma$. Hence proved. 

3.~Let $\mathfrak{M}=(\mathfrak{F},V)$ and $F$ be any filter of $\mathfrak{F}^+$. Let $\Sigma$ be a set of formulas which is finitely satisfiable in $\{J \mid F R_\Box^c J\}$.
By definition of filter-ideal extension, this set is same as  $\{J \mid F \cap  \Box J = \emptyset\}$, where for any $C \subseteq \mathfrak{F}^+$, $\Box C= \{\Box c \mid c \in C\}$. By Lemma \ref{lem:Filter-ideal satisfaction},  for any formula $\varphi$,  $\mathfrak{M}^{\mathsf{FI}}, J\succ \varphi$ iff $(\val{\varphi},\descr{\varphi}) \in J$. Thus, $\Sigma$ is finitely satisfiable in
$\{J \mid F  R_\Box^c J\}$ iff for any finite $\Delta \subseteq \Sigma$, the ideal  generated by the set $\{(\val{\varphi}, \descr{\varphi}) \mid \varphi \in \Delta\}$  (denoted by $Id(\Delta)$) satisfies the condition $\Box Id(\Delta) \cap F =\emptyset $. Let 
\[Y=\bigcup \{ Id(\Delta) \mid 
\Delta \subseteq \Sigma, \, \Delta \,  \mbox{is finite} \}.
\]
It is clear that $F \cap \Box Y= \emptyset$ and $(\val{\varphi}, \descr{\varphi}) \in Y$ for all $\varphi \in \Sigma$. By same proof as in item 1, we have that $Y$ is an ideal. Therefore, $Y$ is an ideal such that $F R_\Box^c Y$ and $\mathfrak{M}^{\mathsf{FI}}, Y \succ \Sigma$. Hence proved.
\end{proof}
%% Bibliography
%% Make sure to use the bibliographystyle aiml22.
\bibliographystyle{aiml}
\bibliography{aiml}

\end{document}